\documentclass[a4paper]{article}
\usepackage{amsthm,amssymb,amsmath,graphicx,psfrag,enumitem,mathrsfs}
\usepackage{algorithm}
\usepackage{tikz}

\usetikzlibrary{backgrounds}
\usetikzlibrary{decorations}
\usetikzlibrary{decorations.pathmorphing}

\newtheorem{definition}{Definition}

\newtheorem{theorem}[definition]{Theorem}
\newtheorem{corollary}[definition]{Corollary}
\newtheorem{lemma}[definition]{Lemma}

\newcommand{\comment}[1]{}
\renewcommand{\comment}[1]{\footnote{{\bf Comment:} #1}}
\newcommand{\N}{\mathbb N}
\newcommand{\R}{\mathbb R}
\newcommand{\Z}{\mathbb Z}

\newcommand{\cP}{\mathcal{P}}
\newcommand{\cQ}{\mathcal{Q}}

\newcommand{\cT}{\mathcal{T}}

\newcommand{\cL}{\mathcal{L}}

\newcommand{\cH}{\mathcal{H}}

\newcommand{\sub}{\subseteq}

\newcommand{\mir}{\preceq_\ell}

\newcommand{\EPP}{Erd\H{o}s-P\'osa property}
\newcommand{\EP}{Erd\H{o}s-P\'osa}

\newcommand{\emtext}[1]{\text{\em #1}}

\newcommand{\sm}{\setminus}

\tikzstyle{hvertex}=[thick,circle,inner sep=0.cm, minimum size=2mm, fill=white, draw=black]
\tikzstyle{svertex}=[hvertex,fill=dunkelgrau]
\tikzstyle{point}=[draw,circle,inner sep=0.cm, minimum size=1mm, fill=black]
\tikzstyle{pointer}=[thick,->,shorten >=2pt,color=hellgrau]
\tikzstyle{hedge}=[draw,very thick]
\colorlet{hellgrau}{black!20!white}
\colorlet{dunkelgrau}{black!50!white}
\colorlet{hellblau}{blue!20!white}
\colorlet{hellrot}{red!40!white}
\definecolor{mossgreen}{rgb}{0.68, 0.87, 0.68}

\title{Erd\H os-P\'osa property for labelled minors: $2$-connected minors}
\author{Henning Bruhn\thanks{Partially supported by DFG, grant no.\  321904558} \and Felix Joos\thanks{The research leading to these results was partially supported by the Deutsche Forschungsgemeinschaft (DFG, German Research Foundation) -- 339933727.} \and Oliver Schaudt}
\date{}

\begin{document}

\maketitle

\begin{abstract}
\noindent
In the 1960s, Erd\H{o}s and P\'osa proved that there is a packing-covering duality for cycles in graphs.
As part of the Graph Minor project, Robertson and Seymour greatly extended this:
there is such a duality for $H$-expansions in graphs if and only if $H$ is a planar graph (this includes the previous result for $H=K_3$).
We consider vertex labelled graphs and minors and provide such a characterisation for $2$-connected labelled graphs~$H$.
In particular, this generalises results of Kakimura, Kawarabayashi and Marx [J. Combin. Theory Ser. B 101 (2011), 378--381]
and Huynh, Joos and Wollan [Combinatorica 39 (2019), 91--133] up to weaker dependencies of the parameters.
\end{abstract} 

\section{Introduction}
The most satisfactory optimisation results are arguably 
the ones that also provide a certificate that the optimum is attained.
An example is Menger's theorem stating that the maximum number of disjoint
paths between two vertex sets is achieved if there is a separator of the same size.
More generally this is  captured by the min cut max flow theorem 
or by the duality principle of linear programming. 

Not always, however, concise certificates for optimality are known or do even exist.
In such a case, an approximate certificate may be available.
There are a few classic examples for this.
One is the triangle removal lemma due to Ruzsa and Szemer\'edi~\cite{RS78}
(for every $\epsilon\in(0,1)$, there is a $\delta>0$
such that every graph on $n$ vertices contains either $\delta n^3$ triangles or $\epsilon n^2$ edges whose deleting makes the graph triangle-free) 
and its generalisations.
The importance of removal lemmas is for example demonstrated by its various applications in number theory, discrete geometry, graph theory and computer science~\cite{CF13}.

Another example is a theorem due to Erd\H{o}s and P\'osa~\cite{EP62}, which also (including its generalisations) has several applications in graph theory and computer science:
every graph $G$ that does not contain $k$ disjoint cycles, 
admits  a vertex set  of size $O(k\log k)$ that meets every cycle.
More generally, we say that a family of graphs $\cH$ has the \emph{\EPP{}} if there exists a 
function $f:\N\to \R_+$
such that for every graph $G$ and every integer $k$,
there exist $k$ disjoint subgraphs in $G$ that are isomorphic to graphs in $\mathcal H$,
or $G$ contains a vertex set $X$ of size $|X|\leq f(k)$ such that every 
subgraph of $G$ isomorphic to a graph in $\mathcal H$ meets $X$.
Thus, the class of cycles has the \EPP.

The \EPP{} has been investigated for numerous graph classes (see~\cite{RT17} for a recent survey).
One of the most striking results is the following due to Robertson and Seymour 
that is a by-product of their Graph Minor project.
It provides another characterisation of planar graphs,
which in fact does not involve any topological arguments.
(Essentially, a graph is an \emph{$H$-expansion} if it can be turned into $H$ by 
a series of edge contractions; see the next section for a formal definition.)

\begin{theorem}[Robertson and Seymour~\cite{RS86}]\label{RSmetathm}
Let $H$ be a graph. The family of $H$-expansions has the Erd\H os-P\'osa property
if and only if $H$ is planar.
\end{theorem}

Observe that this includes the class of cycles (set $H=K_3$).

There are further extensions of the theorem of Erd\H os and P\'osa.
Suppose we specify a set of labelled vertices $S$ in a graph $G$ and now we ask for cycles that contain at least one vertex from $S$ (such cycles are also known as $S$-cycles).
Kakimura, Kawarabayashi and Marx~\cite{KKM11} proved that $S$-cycles also have the \EPP{}~(see~\cite{BJS18,PW12} for further extensions).
Clearly, this is a generalisation because we may take $S=V(G)$.
In~\cite{HJW19}, Huynh, Joos and Wollan extended this to cycles with two labels.

We characterise all labelled $2$-connected graphs $H$ such that the class of labelled $H$-expansions has the \EPP{}.
For simplicity, let us assume for now that every vertex has at most one label
and we define a (sub)graph to be \emph{simply-labelled} if all vertices with a label have the same one.
For a $2$-connected graph $H$, an $H$-expansion (we assume that the branch sets are trees that are connected by at most one edge) is a labelled $H$-expansion if for every vertex $x$ of $H$ with a label $\alpha$  
and each pair $y,z$ of neighbours of $x$,
the unique path from the branch set of $y$ through the branch set of $x$ to the branch set of $z$ contains an internal vertex labelled with $\alpha$.
See Section~\ref{sec:lgraphs} for more details and a brief discussion.

\begin{theorem}\label{thm:2consimple}
Let $H$ be a labelled $2$-connected graph such that each vertex carries at most one label.
Then the labelled $H$-expansions have the \EPP{} if and only if there is an embedding of $H$ in the plane such that the boundary $C$ of the outer 
face contains all labelled vertices, and there are two simply-labelled subpaths $P,Q\subseteq C$
that cover all of~$V(C)$.
\end{theorem}

%We have actually not yet specified what a labelled $H$-expansion is.
There are several ways to define labelled expansions.
We choose a definition such that the resulting labelled minor relation is transitive and we also generalise the results about labelled cycles.
As the precise definition is a bit technical, we defer it to Section~\ref{sec:labels}.
We note that, with a slightly stronger notion of labelled subdivisions,
Liu~\cite{Liu17} proved a half-integral Erd\H os-P\'osa type result for labelled subdivisions.

Theorem~\ref{thm:2consimple} has a number of applications.
It implies the result of Kakimura, Kawarabayashi and Marx that $S$-cycles have the \EPP\ (albeit with a faster growing $f$)
as well as the result due to Huynh, Joos and Wollan that the same is true for cycles with two labels. 
Moreover, more complicated variants of cycles with labelled vertices are covered. For instance, 
the theorem shows that, given a set $S$, the family of cycles that each contain, say, at least~$42$ vertices
from $S$ has the \EPP. Instead of $S$-cycles, we could also consider $S$-$K_4$-subdivisions, 
that is, subdivisions of $K_4$ that each contain at least one vertex from $S$.
As a consequence of our theorem, the set 
of these has the \EPP, too.
Similar statements involving two labels are also covered.

Our main theorem requires the graph $H$ to be $2$-connected. 
This is necessary: if $H$ is not $2$-connected then 
the conclusion of the theorem becomes false;
in particular, there are simply-labelled graphs $H$ such that all labelled vertices belong to the boundary of a single face
but $H$-expansions do not have the \EPP.
We investigate the Erd\H os-P\'osa property for unconnected and merely $1$-connected graphs $H$ in a follow-up 
paper in which we heavily rely on the results of this paper.

At the end of this article, in Sections~\ref{zerosec} and~\ref{infgroupsec}, we discuss 
how Theorem~\ref{thm:2consimple} can further be generalised to include parity and modularity constraints
such that it covers, for instance, even $S$-cycles. 

\section{Labelled graphs and minors}\label{sec:labels}

In this section we introduce several definitions concerning labelled graphs, minors, expansions, walls, and tangles.
All definitions not involving labels are standard and commonly used in the literature.
Most of our notation is standard and in accordance with Diestel~\cite{diestelBook17}.

We start with expansions and minors without labels.
For a graph $H$, a pair $(X,\pi)$ of a graph $X$ and a mapping $\pi:V(H)\cup E(H) \to V(X)\cup E(X)$ is an \emph{$H$-expansion} if 
\begin{enumerate}[label=(\roman*)]
\item $(\pi(u)\cap V(X),\pi(u)\cap E(X))$ 
is an induced tree in $X$ for each $u\in V(H)$ and every vertex in $X$ belongs to exactly one such tree; and
%is a partition of $V(X)$ into vertex-disjoint induced subgraphs of $X$
%such that $\pi(u)$ is a tree for all $u\in V(H)$; and
\item for every two distinct $u,v\in V(H)$
if $u$ and $v$ are adjacent in $H$, there is exactly one $\pi(u)$--$\pi(v)$ edge in $X$,
the edge $\pi(uv)$, and if $u$ and $v$ are not adjacent there is no such edge. 
\end{enumerate}
Often we omit $\pi$ and simply say that $X$ is an $H$-expansion. 
If a graph $G$ contains an $H$-expansion $X$ as a subgraph, we say that $H$ is a \emph{minor} of $G$.
Note that for every vertex $u$ of $H$ the induced subgraph $X[\pi(u)]$ together
with all edges $\pi(uv)$ for $v\in N_H(u)$ 
forms a tree, which we denote by $T^\pi_u$.
We refer to $\pi(u)$ as the \emph{branch set} of $u$.

\subsection{Labelled graphs}\label{sec:lgraphs}

Let us now formally introduce labelled graphs and labelled expansions.
We call a graph $G$ a \emph{labelled graph} if some of its vertices are marked with one or more labels 
from some alphabet $\Sigma$. 
Formally, $G$ is endowed with a function $\ell: V(G)\to \cP(\Sigma)$, 
and we say that a vertex $v$ \emph{is labelled with} $\alpha\in\Sigma$ if $\alpha\in\ell(v)$.
Note that a vertex may have several labels or none at all. 
We also write that a graph $G$ is \emph{$\Sigma$-labelled}.

\begin{figure}[ht]
\centering
\begin{tikzpicture}[scale=0.7]

\tikzstyle{enddot}=[circle,inner sep=0cm, minimum size=10pt,color=hellgrau,fill=hellgrau]
\tikzstyle{markline}=[draw=hellgrau,line width=10pt]

\def\step{0.7}

\begin{scope}[shift={(-6,1)}]
\def\radius{0.5}
\draw[hedge] (0,0) circle (\radius);
\node[svertex] at (120:\radius){};

\node at (1,0){$\preceq_r$};

\begin{scope}[shift={(2,0)}]
\draw[hedge] (0,0) circle (\radius);
\node[hvertex] (a) at (60:\radius){};
\node[svertex] (b) at (50:1.5){};
\draw[hedge] (a) -- (b);
\end{scope}
\end{scope}

\node[svertex] (x) at (0,0){};
\node[hvertex] (a) at (120:\step){};
\node[hvertex] (b) at (60:\step){};
\node[hvertex] (c) at (240:\step){};
\node[hvertex] (d) at (300:\step){};

\draw[hedge] (x) -- (a) -- (b) -- (x);
\draw[hedge] (x) -- (c) -- (d) -- (x);

\begin{scope}[shift={(2,2+0.5*\step)}]
\node[svertex] (xx) at (0,0){};
\node[hvertex] (z) at (\step,0){};
\node[hvertex] (a) at (120:\step){};
\node[hvertex] (b) at (60:\step){};
\node[hvertex] (y) at (0,-\step){};

\begin{scope}[shift={(0,-\step)}]
\node[hvertex] (c) at (240:\step){};
\node[hvertex] (d) at (300:\step){};
\end{scope}

\draw[hedge] (xx) -- (a) -- (b) -- (xx);
\draw[hedge] (y) -- (c) -- (d) -- (y);
\draw[hedge] (xx) edge (y) edge (z); 

\begin{scope}[on background layer]
  \draw[markline,fill=hellgrau] (xx.center) -- (y.center);
  \node[enddot] at (xx){};
  \node[enddot] at (y){};
\end{scope}

\end{scope}

\begin{scope}[shift={(4,0.5*\step)}]
\node[hvertex] (xx) at (0,0){};
\node[svertex] (zz) at (\step,0){};
\node[hvertex] (z) at (2*\step,0){};
\node[hvertex] (a) at (120:\step){};
\node[hvertex] (b) at (60:\step){};
\node[hvertex] (y) at (0,-\step){};

\begin{scope}[shift={(0,-\step)}]
\node[hvertex] (c) at (240:\step){};
\node[hvertex] (d) at (300:\step){};
\end{scope}

\draw[hedge] (xx) -- (a) -- (b) -- (xx);
\draw[hedge] (y) -- (c) -- (d) -- (y);
\draw[hedge] (xx) edge (y) edge (zz); 
\draw[hedge] (zz) -- (z);
\begin{scope}[on background layer]
  \draw[markline,fill=hellgrau] (xx.center) -- (zz.center);
  \node[enddot] at (xx){};
  \node[enddot] at (zz){};
\end{scope}
\end{scope}
 
\node at (2,-0.2) {${\not\preceq}_M$};
\node[rotate=40] at (1,0.7) {$\preceq_M$};
\node[rotate=-40] at (3,0.7) {$\preceq_M$};

\end{tikzpicture}
\caption{Labelled graphs with two different minor relations. Labelled vertices in grey. 
Left: an $S$-cycle as a rooted minor. Right: transitivity fails for naive labelled minor relation}\label{fig:naive-minors}
\end{figure}
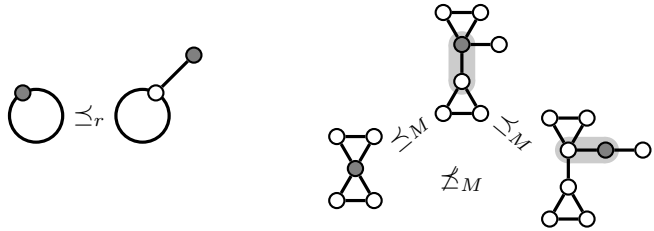

What should it mean that some (labelled) graph has some other graph $H$ as a labelled minor, 
or equivalently, contains a labelled $H$-expansion? 
A natural labelled minor relation has been explored 
before: Wollan~\cite{DBLP:journals/jgt/Wollan08} and Marx, Seymour and Wollan~\cite{MSW17}
treat \emph{rooted minors}, minors with a single label. 
In this setting,  a vertex in a minor is labelled as soon as its branch set contains a labelled vertex (a \emph{root}).
While this definition bears its own merit, it does not capture all structures we want to express.
In particular, it does not capture $S$-cycles: if a graph contains an $S$-cycle 
 as a rooted minor, then it does not necessarily
contain an $S$-cycle as a subgraph; see Figure~\ref{fig:naive-minors}.
Our notion of a labelled minor will be designed to capture $S$-cycles, as well as 
\emph{long} $S$-cycles, $S$-cycles of length at least a fixed length~$\ell$.
These are known to have the \EPP~\cite{BJS18}.

The problem with rooted minors, at least in view of $S$-cycles, is that a branch set may
send out an appendix to pick up a labelled vertex, where this appendix is 
unnecessary for the (unlabelled) minor relation. 
At first sight, the following variant of the definition fixes this issue: 
say a vertex $v$ in a minor is labelled as soon as its branch set contains a labelled vertex and that labelled vertex lies on a path between two edges in the expansion  that connect that branch set to the branch sets of other vertices.
This definition, however, leads to a labelled minor relation that is not transitive, which is 
clearly problematic (see Fig.~\ref{fig:naive-minors})
also because labelled $H$-expansions do not necessarily contain a labelled $H'$-expansion for all subgraphs $H'$ of $H$ (nevertheless such a notion has been considered in~\cite{KM:19} in the context of the \EPP{} for disconnected graphs).
Our notion of a labelled minor is slightly different but transitive 
and hence also closed under taking subgraphs.

\begin{figure}[ht]
\centering
\begin{tikzpicture}
\tikzstyle{hvertex}=[thick,circle,inner sep=0.cm, minimum size=2mm, fill=white, draw=black]
\tikzstyle{svertex}=[hvertex,fill=dunkelgrau]
\tikzstyle{branchset}=[thin,color=hellgrau,fill=hellgrau,circle,minimum size=1.5cm]
\tikzstyle{pathedge}=[hedge,decorate, decoration={random steps,segment length=3pt,amplitude=1pt}]

\tikzstyle{enddot}=[circle,inner sep=0cm, draw, very thick, minimum size=9pt,color=dunkelgrau,fill=hellgrau]
\tikzstyle{markline}=[draw=dunkelgrau,line width=10pt]
\tikzstyle{inner}=[draw=hellgrau,fill=dunkelgrau,line width=8pt]

\def\angle{50}
\def\dist{0.6}

\begin{scope}[rotate=\angle]

\node[svertex] (a) at (0,0) {};
\node[svertex,label=right:$v$] (b) at (0,0.7) {};
\node[hvertex] (c) at (-0.5,1.4) {};
\node[svertex,label=right:$u$] (d) at (0.5,1.4) {};
\node[hvertex] (e) at (0,2.1) {};

\draw[hedge] (a) -- (b);
\draw[hedge] (b) -- (c);
\draw[hedge] (b) -- (d);
\draw[hedge] (c) -- (d);
\draw[hedge] (e) -- (c);
\draw[hedge] (e) -- (d);
\draw[hedge] (e) -- (b);
\end{scope}

\begin{scope}[shift={(6,-0.5)},rotate=\angle]

\begin{scope}[scale=2.5,on background layer]
\node[branchset] (A) at (0,0) {};
\node[branchset,label=right:$T^\pi_v$, thick, color=dunkelgrau] (B) at (0,0.7) {};
\node[branchset] (C) at (-0.4,1.25) {};
\node[branchset,label=right:$\pi(u)$] (D) at (0.4,1.25) {};
\node[branchset] (E) at (0,1.8) {};
\end{scope}

\begin{scope}[shift={(A.center)}]
\node[svertex] (a0) at (0,-0.5) {};
\node[hvertex] (a1) at (0,\dist) {};
\end{scope}

\begin{scope}[shift={(B.center)}]
%\node[hvertex] (b0) at (0,0) {};
\node[hvertex] (b1) at (0,-\dist) {};
\node[svertex] (b2) at (80:\dist) {};
\node[hvertex] (b3) at (150:\dist) {};
\node[svertex] (cb) at (0,0){};
\end{scope}

\begin{scope}[shift={(C.center)}]
%\node[hvertex] (b0) at (0,0) {};
\node[hvertex] (c1) at (-70:\dist) {};
\node[hvertex] (c2) at (0:\dist) {};
\node[hvertex] (c3) at (70:\dist) {};
\coordinate (cc) at (180:0.5*\dist);
\end{scope}

\begin{scope}[shift={(D.center)}]
%\node[hvertex] (b0) at (0,0) {};
\node[hvertex] (d1) at (180-70:\dist) {};
\node[hvertex] (d2) at (180:\dist) {};
\node[hvertex] (d3) at (180+70:\dist) {};
\coordinate (cd) at (0:0.5*\dist);
\path (cd) to coordinate[midway] (dx) (d1.center);
%\path (cd) to coordinate[midway] (dy) (d2.center);
\path (cd) to  coordinate[midway] (dz) (d3.center);
\node[svertex]  at (dx) {};
%\node[svertex]  at (dy) {};
\node[svertex]  at (dz) {};
\end{scope}

\begin{scope}[shift={(E.center)}]
%\node[hvertex] (b0) at (0,0) {};
\node[hvertex] (e1) at (-30:\dist) {};
\node[hvertex] (e2) at (180+30:\dist) {};
\node[hvertex] (e3) at (-90:\dist) {};
\coordinate (ce) at (90:0.5*\dist);
\end{scope}

\draw[hedge] (a1) -- (b1);
\draw[hedge] (b3) -- (c1);
\draw[hedge] (b2) -- (d3);
\draw[hedge] (b2) -- (e3);
\draw[hedge] (c2) -- (d2);
\draw[hedge] (d1) -- (e1);
\draw[hedge] (c3) -- (e2);

\node at (0.95,1.95) {$\pi(uv)$};

\begin{scope}[on background layer]
\node[enddot] at (d3.center){};
\draw[markline] (d3.center) -- (b2.center);
\draw[inner] (d3.center) -- (b2.center);

\node[enddot] at (e3.center){};
\draw[markline] (b2.center) -- (e3.center);
\draw[inner] (b2.center) -- (e3.center);

\node[enddot] at (a1.center){};
\draw[markline] (b1.center) -- (a1.center);
\draw[inner] (b1.center) -- (a1.center);

\node[enddot] at (c1.center){};
\draw[markline] (b3.center) -- (c1.center);
\draw[inner] (b3.center) -- (c1.center);

\draw[fill=hellgrau,color=hellgrau] (b2.center) circle (0.22cm);

\node[branchset,minimum size=1.4cm] at (B){};

\end{scope}

\begin{scope}[on background layer]
\draw[pathedge] (a0.center) to (a1.center);
\draw[pathedge] (cb) to (b1.center);
\draw[pathedge] (cb) to (b2.center);
\draw[pathedge] (cb) to (b3.center);
\draw[pathedge] (cc) to (c1.center);
\draw[pathedge] (cc) to (c2.center);
\draw[pathedge] (cc) to (c3.center);
\draw[pathedge] (cd) to (dx) to (d1.center);
%\draw[pathedge] (cd) to (dy) to (d2.center);
\draw[pathedge] (cd) to (d2.center);
\draw[pathedge] (cd) to (dz) to (d3.center);
\draw[pathedge] (ce) to (e1.center);
\draw[pathedge] (ce) to (e2.center);
\draw[pathedge] (ce) to (e3.center);
\end{scope}

\end{scope}

\node at (1,1) {{\large $\mir$}};

\end{tikzpicture}
\caption{A labelled expansion (labelled vertices in grey)}\label{lblexpfig}
\end{figure}
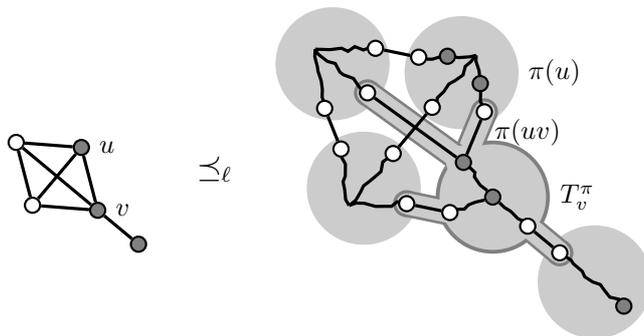

Fix some alphabet $\Sigma$ and let $H$ be a $\Sigma$-labelled graph. A pair $(X,\pi)$ of a 
labelled graph $X$ and a mapping $\pi:V(H)\cup E(H) \to V(X)\cup E(X)$
is a \emph{labelled $H$-expansion} if
\begin{enumerate}[label=(\roman*)]
\item $(X,\pi)$  is an $H$-expansion; and 
\item if $v\in V(H)$ is labelled with~$\alpha$ and if $T^\pi_v$ is not an isolated vertex, then 
every non-trivial leaf-to-leaf path in $T^\pi_v$ contains a vertex contained in $\pi(u)$ that is labelled with~$\alpha$.
\end{enumerate}
Observe that $T^\pi_v$ may only be an isolated vertex if $v$ is an isolated vertex.
Intuitively, the definition says that if $u$ and $v$ are neighbors of some vertex $w$ in $H$, the direct path from $\pi(u)$ to $\pi(v)$ through $\pi(w)$ contains vertices of every label in $\ell(v)$.

Again, if the mapping $\pi$ is clear from the context, we may simply call $X$
itself a labelled $H$-expansion. 
If a labelled graph $G$ contains a labelled $H$-expansion as a subgraph, 
$H$ is a \emph{labelled} minor of $G$.
We write $H\mir G$ for short.

Let us first show that this definition yields a transitive minor relation.
To this end, 
we say that a labelled $H$-expansion $(X,\pi)$ is \emph{minimal}
if for all $u\in V(H)$ the following holds:
\begin{itemize}
	\item If $d_H(u)\geq 2$, then every leaf of $\pi(v)$ is contained in some $\pi(uv)$ for some $v\in N_H(u)$; and
	\item if $d_H(u)\leq 1$, then $\pi(u)$ is a path and if $d_H(u)= 1$, then an endvertex of this path is contained in $\pi(uv)$ where $v$ is the unique neighbour of $u$.
\end{itemize}
It is easy to see that every labelled $H$-expansion $(X,\pi)$ contains as a subgraph a minimal $H$-expansion $(X',\pi')$
such that $\pi'(uv)=\pi(uv)$ for all $uv\in E(H)$ and $\pi'(u)$ is a subtree of $\pi(u)$ for all $u\in V(H)$.

\begin{lemma}\label{lem:transitivity}
Let $A,B,C$ be labelled graphs such that $A \mir B$ and $B\mir C$. 
Then also $A\mir C$.
\end{lemma}
\begin{proof}
Observe first that whenever a graph $G$ contains a labelled $H$-expansion of a graph $H$,
then $G$ also contains a labelled $H'$-expansion for any subgraph $H'$ of $H$.

Hence we may assume that $(B,\beta)$ is a minimal labelled $A$-expansion, 
and that $(C,\gamma)$ is a minimal labelled $B$-expansion. 
Define $$\pi(a)=\bigcup_{b\in V(\beta(a))}\gamma(b) \cup \bigcup_{bb'\in E(\beta(a))}\gamma(bb')$$ for every $a\in V(A)$,
and  set $\pi(aa')=\gamma(\beta(aa'))$ for all $aa'\in E(A)$.
Forgetting the labels, 
it is a standard task to check that  $(C,\pi)$ is an $A$-expansion. 
Thus, it remains to verify condition~(ii) in the definition of labelled expansions.

For this, let $a\in V(A)$ be labelled with~$\alpha$,
and let $P=u\ldots v$ be a leaf-to-leaf path in~$T^\pi_a$.
Since $B$ and $C$ are minimal,
$u$ (resp.~$v$) either does not belong to $\pi(a)$ or $d_A(a)\leq 1$.
Observe that $T^\pi_a=\bigcup_{b\in\beta(a)}T^\gamma_b$.
Then $P$ defines a leaf-to-leaf path $P'$ in $T^\beta_a$ (if $d_A(a)\leq 1$, then $P'=T_a^{\beta}$).
The path $P'$ contains 
a vertex $b^*\in \beta(a)$ that is labelled with~$\alpha$
 as $(B,\beta)$ is a labelled $A$-expansion.
The path $Q=P\cap T^\gamma_{b^*}$ is, in $T^\gamma_{b^*}$, a leaf-to-leaf path as well.
Since $b^*$ is labelled with~$\alpha$ it follows that  $Q$ contains a vertex $c^*$ in $\gamma(b^*)$
that is labelled with~$\alpha$ as well. 
Since $c^*\in V(\gamma(b^*))\subseteq V(\pi(a))$
we have found a vertex in $\pi(a)$ on $P$ that is labelled with~$\alpha$, as desired.
\end{proof}

The definition of a labelled graph or expansion
allows for vertices to receive two or more labels, and this is actually 
helpful in the proofs.
However, 
our main result, Theorem~\ref{thm:2consimple}, requires the vertices in the graph $H$ to have at most one label. 
This is mostly because we favour main theorems with simple statements.
Allowing doubly-labelled vertices in $H$ complicates matters somewhat. While we can (and will) handle 
these complications, the resulting statement becomes more complex, and less attractive (see Theorem~\ref{thm:2con}).

\section{Tangles}

The concept of a tangle plays a key role in this paper.
We start with the definition and explain how a minimal counterexample for the \EPP{} of a certain family of graphs
naturally yields a tangle.
We then introduce walls and recall how tangles are linked to walls.
In Section~\ref{sec:linkages}, we introduce linkages and 
in Section~\ref{sec:flatwall}, we state the key tool for our proof.

\subsection{Definition}

An ordered pair $(A,B)$ of edge-disjoint subgraphs of $G$ that partition $E(G)$ 
is a \emph{separation}.
The \emph{order} of the separation is~$|V(A)\cap V(B)|$.

A \emph{tangle} of order~$r$ in a graph $G$ is a set $\mathcal T$
of tuples $(A,B)$ so that the following assertions hold.
\begin{enumerate}[label=(T\arabic*)]
\item\label{T1} 
Every tuple $(A,B)\in \cT$ is a separation of order less than $r$.
\item\label{T2}
For all separations $(A,B)$ of $G$ of order less than $r$,
exactly one of $(A,B)$ and $(B,A)$ lies in $\cT$.
\item\label{T3} $V(A)\neq V(G)$ for all $(A,B)\in \cT$.
\item\label{T4} $A_1\cup A_2\cup A_3\neq G$ for all $(A_1,B_1),(A_2,B_2),(A_3,B_3)\in\mathcal T$.
\end{enumerate}

For any tangle $\cT$ and $(A,B)\in \cT$,
we refer to $A$ as the \emph{$\cT$-small} side of the separation $(A,B)$.
Suppose $\cT$ has order $r\geq 3$ and let $X$ be a vertex set of size at most $r-2$.
Then $G-X$ contains a unique block $U$
such that $V(U)\cup X$ is not contained in any $\cT$-small side of a separation in $\cT$.
We call the block~$U$  the \emph{$\cT$-large block} of $G-X$.

\subsection{Tangles and the \EPP}

The concept of tangles goes very well together with the \EPP.
To see this we first introduce the notion of a minimal counterexample.
Suppose $\cH$ is a family of graphs, $G$ is a graph, and $k\in \N$.
We say that $G$ is \emph{$\mathcal H$-free} if no subgraph of $G$ lies in $\mathcal H$.
We say the pair $(G,k)$ is a \emph{counterexample} to the function $f:\N\to \R_+$ being 
an \EP{} function for the family $\cH$ if $G$ does contain neither $k$ disjoint copies of graphs in $\cH$
nor a set $X\subseteq V(G)$ of size at most $f(k)$ such that $G-X$ is $\cH$-free.
%the following statements hold.
%\begin{enumerate}[label=(MC\arabic*)]
%	\item The graph $G$ does contain neither $k$ disjoint copies of graphs in $\cH$
%	nor a set $X\subseteq V(G)$ of size at most $f(k)$ such that $G-X$ is $\cH$-free.
%	\item For every $k'<k$, the graph $G$ contains $k'$ disjoint copies of graphs in $\cH$
%	or a set $X\subseteq V(G)$ of size at most $f(k')$ such that $G-X$ is $\cH$-free.
%\end{enumerate}
We extend this definition to the labelled case in a straightforward way: 
$\cH$ is a labelled family of graphs, $G$ is a labelled graph, and $G$ does neither contain $k$ disjoint copies of labelled graphs in $\cH$ nor a set 
$X\subseteq V(G)$ of size at most $f(k)$ such that $G-X$ is $\cH$-free.

The following lemma shows that every minimal counterexample has a somewhat canonical tangle
which indicates where the copies of the graphs in $\cH$ lie.
Essentially the same lemma was proven by Wollan in~\cite{Wol11} and restated and adapted in~\cite{HJW19}.
We include the short proof for completeness.

\begin{lemma}\label{largetanglelem}
Suppose $\cH$ is a family of connected
labelled graphs.
Suppose $(G,k)$ is a counterexample to the function $f:\N\to \R_+$ being 
an \EP{} function for $\cH$ with $k$ chosen minimal over all such counterexamples.
Suppose that $t\leq \min\{ f(k)-2f(k-1) , f(k)/3\}$.
Let $\cT$ be the collection of all separations $(A,B)$ of order less than $t$
such that $B$ contains a subgraph that lies in $\cH$.
Then $\cT$ is a tangle.
\end{lemma}
\begin{proof}
Observe that $k\geq 2$.
To verify that $\cT$ is a tangle, we only need to check \ref{T2}--\ref{T4}.
Let $(A,B)$ be a separation of $G$ of order less than $t$.
We claim that one of $A-B$ and $B-A$ contains a graph of $\cH$.
If not, set $X=V(A\cap B)$ and observe that $G-X$ is $\cH$-free,
which is impossible as $|X|<t<f(k)$.

Next, suppose that both $A$ and $B$ contain a copy of a graph in $\cH$.
Then, neither of $A-B$ and $B-A$ can contain $k-1$ copies of graphs in $\cH$ 
as $(G,k)$ is a counterexample.
Hence there are a sets $X_A\sub V(A)$, $X_B\sub V(B)$, each of size at most $f(k-1)$,
such that both $A-(V(B)\cup X_A)$ and $B-(V(A)\cup X_B)$ are $\cH$-free.
But then $G-(X_A\cup X_B \cup (V(A)\cap V(B))$ is $\cH$-free
(recall that the graphs in $\cH$ are connected), which is 
impossible as 
\[
|X_A\cup X_B \cup (V(A)\cap V(B))|\leq 2f(k-1)+t\leq f(k).
\]
Therefore, \ref{T2} holds. For~\ref{T3}, observe that $B-A=\emptyset$ if $V(A)=V(G)$,
which clearly implies that $B-A$ cannot contain any graph from $\cH$.

Finally, suppose there are three separations $(A_1,B_1)$, $(A_2,B_2)$, $(A_3,B_3)\in\cT$
such that $A_1\cup A_2 \cup A_3=G$.
Let $X= \bigcup_{i\in [3]}(V(A_i)\cap V(B_i))$, and observe that $|X|\leq 3t\leq f(k)$.
Then, any graph in $\cH$ that is disjoint from $X$ 
must lie in $\bigcap_{i=1}^3B_i-A_i=\emptyset$. (Again, we use here that the graphs in $\cH$ 
are connected.) 
Thus,
$G-X$ is $\cH$-free, which is again a contradiction.
Therefore,~\ref{T4} holds and $\cT$ is a tangle.
\end{proof}

\subsection{Walls}

Let $[r]$ denote the set $\{1,\ldots,r\}$.
The \emph{$r\times s$-grid}, 
$r,s \ge 2$, 
is the graph on the vertex set $[r]\times [s]$ where a vertex $(i,j)$ is adjacent to a vertex $(i',j')$ if and only if $|i - i'| + |j - j' | = 1$.
An \emph{elementary $r$-wall} is the graph obtained from the $2(r+1) \times (r+1)$-grid by deleting all edges of the form $(2i - 1,2j - 1)(2i - 1,2j)$, where $i \in [r+1]$ and $j \in [\lceil r/2 \rceil]$, and also all edges of the form $(2i,2j)(2i,2j + 1)$,
where $i \in [r+1]$ and $j \in [\lfloor (r-1)/2 \rfloor]$, and then deleting the two vertices of degree~$1$.
An elementary $8$-wall is depicted in Figure~\ref{fig:6-wall} (where we assume that first coordinate increases from left to right and the second coordinate increases from bottom to top).

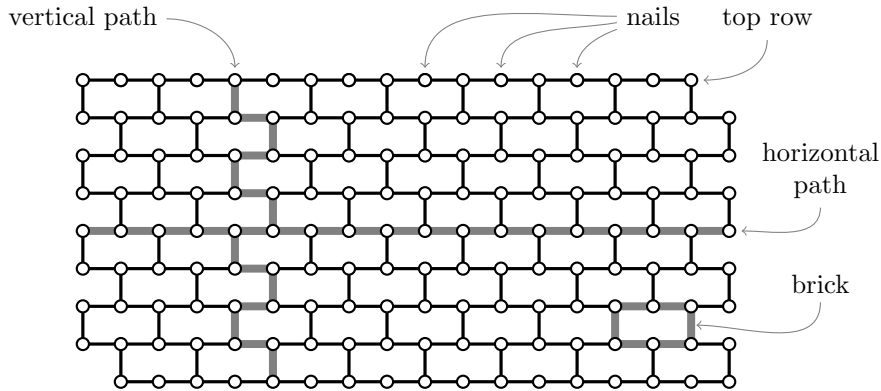
\begin{figure}[ht]
\centering
\begin{tikzpicture}
\tikzstyle{hvertex}=[thick,circle,inner sep=0.cm, minimum size=1.6mm, fill=white, draw=black]
\tikzstyle{marked}=[line width=3pt,color=dunkelgrau]
\tikzstyle{point}=[thin,->,shorten >=2pt,color=dunkelgrau]

\def\wallheight{8}
\def\brickheight{0.5}

\pgfmathtruncatemacro{\lastrow}{\wallheight}
\pgfmathtruncatemacro{\penultimaterow}{\wallheight-1}
\pgfmathtruncatemacro{\lastrowshift}{mod(\wallheight,2)}
\pgfmathtruncatemacro{\lastx}{2*\wallheight+1}

\draw[hedge] (\brickheight,0) -- (2*\wallheight*\brickheight+\brickheight,0);
\foreach \i in {1,...,\penultimaterow}{
  \draw[hedge] (0,\i*\brickheight) -- (2*\wallheight*\brickheight+\brickheight,\i*\brickheight);
}
\draw[hedge] (\lastrowshift*\brickheight,\lastrow*\brickheight) to ++(2*\wallheight*\brickheight,0);

\foreach \j in {0,2,...,\penultimaterow}{
  \foreach \i in {0,...,\wallheight}{
    \draw[hedge] (2*\i*\brickheight+\brickheight,\j*\brickheight) to ++(0,\brickheight);
  }
}
\foreach \j in {1,3,...,\penultimaterow}{
  \foreach \i in {0,...,\wallheight}{
    \draw[hedge] (2*\i*\brickheight,\j*\brickheight) to ++(0,\brickheight);
  }
}

% vertical path
\def\colind{5}
\foreach \j in {2,4,6}{
  \draw[marked] (\colind*\brickheight,\j*\brickheight-2*\brickheight) -- ++ (0,\brickheight) -- ++(-\brickheight,0) -- ++(0,\brickheight) -- ++(\brickheight,0);
}
\draw[marked] (\colind*\brickheight,6*\brickheight) -- ++ (0,\brickheight) -- ++(-\brickheight,0) -- ++(0,\brickheight);

\def\rowind{4}
\foreach \i in {1,...,\lastx}{
  \draw[marked] (\i*\brickheight-\brickheight,\rowind*\brickheight) -- ++(\brickheight,0);
}

\draw[marked] (2*\wallheight*\brickheight,1*\brickheight) -- ++(0,\brickheight) coordinate[midway] (brx)
-- ++(-2*\brickheight,0)
-- ++(0,-\brickheight) -- ++(2*\brickheight,0);

\foreach \i in {1,...,\lastx}{
  \node[hvertex] (w\i w0) at (\i*\brickheight,0){};
}
\foreach \j in {1,...,\penultimaterow}{
  \foreach \i in {0,...,\lastx}{
    \node[hvertex] (w\i w\j) at (\i*\brickheight,\j*\brickheight){};
  }
}
\foreach \i in {1,...,\lastx}{
  \node[hvertex] (w\i w\lastrow) at (\i*\brickheight+\lastrowshift*\brickheight-\brickheight,\lastrow*\brickheight){};
}

\foreach \i in {2,4,...,\lastx}{
  \node[hvertex,fill=white] (w\i w\lastrow) at (\i*\brickheight+\lastrowshift*\brickheight-\brickheight,\lastrow*\brickheight){};
}

\node[anchor=mid] (tr) at (\lastx*\brickheight+0.5,\wallheight*\brickheight+0.8){top row};
\draw[point,out=270,in=0] (tr) to (w\lastx w\wallheight);

\node[anchor=mid] (nails) at (\lastx*\brickheight-1,\wallheight*\brickheight+0.8){nails};
\draw[point,out=180,in=90] (nails) to (w10w\wallheight);
\draw[point,out=190,in=90] (nails) to (w12w\wallheight);
\draw[point,out=200,in=90] (nails) to (w14w\wallheight);

\node[anchor=mid] (vp) at (0,\wallheight*\brickheight+0.8){vertical path};
\draw[point,out=0,in=90] (vp) to (w\colind w\wallheight);

\node[align=center] (hp) at (\lastx*\brickheight+1.2,\rowind*\brickheight+0.8){horizontal\\ path};
\draw[point,out=270,in=0] (hp) to (w\lastx w\rowind);

\node[align=center] (br) at (\lastx*\brickheight+1.2,1*\brickheight+0.8){brick};
\draw[point,out=270,in=0] (br) to (brx);

\end{tikzpicture}
\caption{An elementary $8$-wall}\label{fig:6-wall}
\end{figure}

An \emph{$r$-wall} or simply a \emph{wall} is a subdivision $W$ of an elementary $r$-wall $Z$.
In $Z$ we define the path $P^{(h)}_{j-1}$ for $j\in [r+1]$ as the path on vertices
$ij$ for $i\in [2(r+1)]$ (where we note that $P^{(h)}_0$ as well as $P^{(h)}_r$ are missing
the first or last of these vertices as these are not present in $Z$). 
The paths $P^{(h)}_0,\ldots, P^{(h)}_r$, which are pairwise disjoint, 
are the \emph{horizontal paths} of $Z$. There are also $r+1$ pairwise disjoint 
$P^{(h)}_0$--$P^{(h)}_r$-paths in $Z$, the \emph{vertical paths} $P^{(v)}_0,\ldots,P^{(v)}_r$ of $Z$. 
The path $P^{(h)}_r$ is also called the \emph{top row} of $Z$.
 The vertices of degree~$2$
in the top row are the \emph{nails} of $Z$.
Any $6$-cycle in $Z$ is a \emph{brick} of $Z$.

We keep using the same concepts for walls as for elementary walls. That is, we
will talk about vertical and horizontal paths of $W$, 
and mean the paths that arise from subdividing the corresponding
paths in the elementary wall. A bit of care has to be applied when 
it comes to nails, as there are several choices of vertices in $W$ that correspond
to the (uniquely defined) nails in $Z$. But here, if necessary, we assume that the 
wall $W$ comes with a fixed choice of nails, which allows us to speak about \emph{the} nails of $W$.

Let $s\leq t$.
An \emph{$s$-subwall} $W'$ of a $t$-wall $W$ is subgraph of $W$ that is an $s$-wall
and such that every 
horizontal (vertical) path of $W'$ is a subpath of a unique horizontal (vertical) path of $W$.

\subsection{Tangles and Walls}

We collect more facts about tangles and walls. For more details and proofs see Robertson and Seymour~\cite{RS86}.

Let $\mathcal T$ be a tangle of order $r$, and let $s\leq r$. Let
$\mathcal T'$ be the 
 subset of those $(A,B)\in\mathcal T$ that are
separations of order less than $s$. Then $\mathcal T'$ is again a tangle, the
\emph{truncation of $\mathcal T$ to order~$s$}.

Let $\mathcal{T}$ be a tangle of order $r$ in a graph $H$ and assume that $H$ is a minor of a graph $G$.
We define a tangle $\mathcal{T}_H$ in $G$ \emph{induced by} $H$ as follows.
Let $(C, D)$ be a separation in $G$ of order less than $r$, and let $C_H$
be the induced subgraph of $H$ on all vertices whose branch set in $G$
intersects $C$, and define $D_H$ in the analogous way. Then every edge in $H$ lies
 in $C_H$ or in $D_H$ as  otherwise there would be an edge in $G$ between $C-D$ and $D-C$. 
Moreover, since every branch set that meets $C$ as well as $D$ also contains
a vertex in $C\cap D$, it follows that $|V(C_H\cap D_H)|\leq |V(C\cap D)|$. 
Thus, if we split up the common edges of $C_H$ and $D_H$ we obtain a separation $(C_H,D_H)$
of $H$ of order less than $r$.  
Therefore, either $(C_H, D_H) \in \mathcal{T}$ or $(D_H, C_H) \in \mathcal{T}$ and we then put $(C, D)$ resp.~$(D, C)$ into $\mathcal{T}_H$.
That $\cT_H$ is indeed a tangle was shown by Robertson and Seymour~\cite{RS91}.

Beside the tangle induced by the copies of a certain family $\cH$ of graphs
in a minimal counterexample for the \EPP,
we consider two further tangles.

\begin{lemma}[Robertson and Seymour~\cite{RS91}]\label{lem:KnTangle}
Suppose $n\geq 3$, $t=\lceil \frac{2n}{3} \rceil$, and $\cT$ is the set of all $(t-1)$-separations $(A,B)$ of $K_n$ such that $V(B)=V(K_n)$.  
Then $\cT$ is a tangle. 
\end{lemma}

For a $K_t$-expansion $\pi$,
we refer to $\cT_\pi$ as the tangle induced by the tangle in $K_t$ that is described in Lemma~\ref{lem:KnTangle}.

\begin{lemma}[Robertson and Seymour~\cite{RS91}]\label{lem:WallTangle}
Suppose $t\geq 2$
and $W$ is a $t$-wall.
Let $\cT_W$ be the set of all $t$-separations $(A,B)$ of $W$ such that $B$ contains an entire horizontal path.
Then $\cT_W$ is a tangle of order $t+1$.
\end{lemma}

We also need the converse direction, namely that a tangle of large order forces the existence of a large wall.

\begin{theorem}[Robertson and Seymour~\cite{RS91}]\label{lemma: tangle wall}
For every positive integer $t$, there is an integer $T(t)$
such that if $G$ is a graph that has a tangle $\cT$ of order $T(t)$,
then there is a $t$-wall $W$ in $G$ such that $\cT_W$ is a truncation of $\cT$.
\end{theorem}

\subsection{Linkages}\label{sec:linkages}

Let $G$ be a graph, $k\in\N$, and let $A,B$ be subgraphs, or vertex sets, of $G$.
An \emph{$A$--$B$-path} is a path from some $a \in A$ to some $b \in B$ that is internally disjoint from $A \cup B$.
Moreover, an \emph{$A$-path} is an $A$--$A$-path with at least one edge; if the path consist
of a single edge, then this edge must not lie in $A$.

Let $W$ be a wall with nails $N$. A \emph{$W$-linkage} $\cL$ of order $k$, or simply a \emph{linkage}, 
is a set of $k$ disjoint $W$-paths with first and last vertices in $N$. 
The top row of $W$ defines a linear order $\leq$ (in fact two; we pick one) on the nails.
Consider two paths $P,Q$ in $\cL$, and let the endvertices of $P$ be $p_1 < p_2$, and let 
the endvertices of $Q$ be $q_1<q_2$. By symmetry, we may assume that $p_1<q_1$.
Then $P$ and $Q$ are \emph{in series} if $p_2 < q_1$; 
they are \emph{nested} if $p_1 < q_1 < q_2 < p_2$;
and they are \emph{crossing} if $p_1 < q_1 < p_2 < q_2$; see Figure~\ref{linkagefig}.
The linkage $\cL$ is \emph{in series}, \emph{nested}, or \emph{crossing} if all paths in $\cL$ are mutually in series, nested, or crossing.
We call $\cL$ \emph{pure} if it is in series, nested, or crossing.

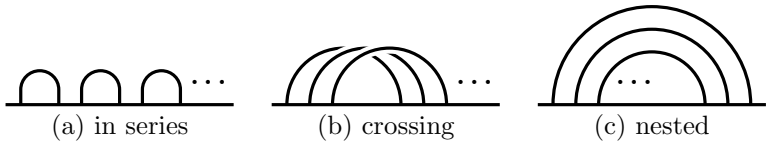
\begin{figure}[ht]
\centering
\begin{tikzpicture}

\tikzstyle{oben}=[line width=1.5pt,double distance=1.2pt,draw=white,double=black]

\begin{scope}
\def\radius{0.25}
\def\step{0.3}
\draw[hedge] (0,0) -- (3,0);
\draw[hedge] (0.2,0) -- ++(0,0.2) arc (180:0:\radius) -- ++(0,-0.2);
\draw[hedge] (0.2+2*\radius+\step,0) -- ++(0,0.2) arc (180:0:\radius) -- ++(0,-0.2);
\draw[hedge] (0.2+4*\radius+2*\step,0) -- ++(0,0.2) arc (180:0:\radius) -- ++(0,-0.2);
\node at (2.7,0.3) {{\bf \dots}};
\node at (1.5,-0.3) {(a) in series};
\end{scope}

\begin{scope}[shift={(3.5,0)}]
\draw[oben] (0.2,0) arc (180:0:0.75);
\draw[oben] (0.2+0.3,0) arc (180:0:0.75);
\draw[oben] (0.2+2*0.3,0) arc (180:0:0.75);
\draw[hedge] (0,0) -- (3,0);
\node at (2.7,0.3) {{\bf \dots}};
\node at (1.5,-0.3) {(b) crossing};
\end{scope}

\begin{scope}[shift={(7,0)}]
\draw[hedge] (0,0) -- (3,0);
\draw[hedge] (0.2,0) arc (180:0:1.3);
\draw[hedge] (0.2+0.3,0) arc (180:0:1.0);
\draw[hedge] (0.2+2*0.3,0) arc (180:0:0.7);
\node at (1.3,0.3) {{\bf \dots}};
\node at (1.5,-0.3) {(c) nested};
\end{scope}

\end{tikzpicture}
\caption{The three types of pure linkages}\label{linkagefig}
\end{figure}

Assume $W$ to be contained in a $\Sigma$-labelled graph, and let $\alpha\in \Sigma$.
A $W$-linkage $\cL$ is called $\alpha$-\emph{clean}\footnote{We adapt here a notion introduced by Huynh et al.~\cite{HJW19} to the labelled setting. To keep notation simple, we have slightly weakened it.}
if
\begin{itemize}
	\item $\cL$ is pure, and
	\item every path in $\cL$ contains a vertex of label $\alpha$.
\end{itemize}
Moreover, let $(\cP,\cQ)$ be a partition of a   $W$-linkage $\mathcal P\cup\mathcal Q$.
We call $(\cP,\cQ)$ a pair of $(\alpha,\beta)$-\emph{clean} $W$-linkages if
\begin{itemize}
	\item $\cP$ is $\alpha$-clean  and if $\cQ$ is  $\beta$-clean,
	\item $|\cP| = |\cQ|$, and
	\item for all $P,P'\in \cP$ and $Q\in \cQ$ with endvertices $p_1<p_2$, $p_1'<p_2'$ and $q_1<q_2$,
	we have $q_1,q_2\not\in [p_1,p_1']\cup [p_2,p_2']$. Here, $[p_1,p_1']$ is the set of all nails $v$
with $p_1\leq v\leq p'_1$, and $[p_2,p_2']$ is defined similarly.
\end{itemize}

\subsection{Flat walls}\label{sec:flatwall}

In their so-called \emph{flat wall theorem} Robertson and Seymour~\cite{RS95} proved that every graph with a huge wall contains a large clique-minor or
a large \emph{flat} wall, a wall that lies in a \emph{nearly} planar part of the graph.
Huynh, Joos, and Wollan~\cite{HJW19} extended the theorem to graphs whose 
edges are labelled with elements from two groups.
We present below a version of the theorem that is adapted to labelled graphs.\footnote{
To derive the stated version for their theorem, we choose for both groups $(\Z,+)$. 
For an arbitrary ordering $e_1,e_2,\ldots$ of the edges of $G$,
we assign to $e_i$ the group value $2^i$ in the first (second) coordinate
if one of the endpoints of $e_i$ is labelled with $\alpha$ ($\beta$) and otherwise $0$.
Then every cycle that is non-zero in both coordinates corresponds to a cycle that contains both a vertex labelled $\alpha$ and a vertex labelled $\beta$.
}
For our purposes it is not important that the wall is flat,
so we simply drop the condition.

We need a little bit more notation before we can state our main tool, the result
of Huynh et al.
We define a sort of doubly-labelled expansion of a complete graph. For technical 
reasons, we weaken the definition of an expansion slightly.
Let $\pi:V(K_n)\cup E(K_n)\to V(G)\cup E(G)$ for some graph $G$, and let $\alpha,\beta$
be two labels.
We say $\pi$ is a \emph{$(\alpha,\beta)$-thoroughly labelled} (pseudo) $K_n$-expansion in $G$
if
\begin{itemize}
	\item $\pi(x)$ is a tree for every vertex $x$ of $K_n$,
	\item $\pi(xy)$ is a set of at most two edges joining $\pi(x)$ and $\pi(y)$, and
	\item for every $\gamma\in\{\alpha,\beta\}$ and every triple $x,y,z$ of vertices of $K_n$, 
	there exist $e_{ab}\in\pi(ab)$ for each $ab\in\{xy,xz,yz\}$ such that
	$\pi(x)\cup\pi(y)\cup\pi(z)\cup e_{xy}\cup e_{xz}\cup e_{yz}$
	contains a vertex with label~$\gamma$.
\end{itemize}
Although, technically, these pseudo expansions are not expansions in the strict sense
we defined earlier, we will simply call them 
{$(\alpha,\beta)$-thoroughly labelled} $K_n$-expansions, which is already long enough.

For walls we have an analogous concept. A wall $W$ is \emph{thoroughly $\alpha$-labelled}
if every  brick contains a vertex with label $\alpha$, and the wall is 
\emph{thoroughly $(\alpha,\beta)$-labelled} if every brick contains a vertex with 
label $\alpha$ and a vertex with label~$\beta$.

\begin{theorem}[Huynh, Joos, and Wollan~\cite{HJW19}]\label{thm:simpleflatwall}
For every $t\in \N$,
there exists an integer $t'$
such that if $G$ is an $(\alpha,\beta)$-labelled graph 
that contains a $t'$-wall $W$ then  one of the following statements holds.
\begin{enumerate}[label=\rm (\roman*)]
	\item\label{case1} There is an $(\alpha,\beta)$-thoroughly labelled $K_t$-expansion $\pi$ in $G$
	such that $\cT_{\pi}$ is a truncation of $\cT_W$.
	\item\label{case2} There is a $100t$-wall $W_0$
	such that $\cT_{W_0}$ is a truncation of $\cT_W$ and 
	\begin{enumerate}[label=\rm (\alph*)]
		\item\label{case2a} $W_0$ is $(\alpha,\beta)$-thoroughly labelled,
		\item\label{case2b} for some $\gamma\in \{\alpha,\beta\}$, 
		the wall $W_0$ is $\gamma$-thoroughly labelled and
		has an $(\{\alpha,\beta\}\sm \gamma)$-clean $W_0$-linkage of size $t$, or
		\item\label{case2c} $W_0$ has a pair of $(\alpha,\beta)$-clean $W_0$-linkages of size $t$.
	\end{enumerate}
	\item\label{case3} For some $\gamma\in \{\alpha,\beta\}$, 
there is a set $Z$ such that $|Z|< t'$ and the unique $\cT_W$-large block of $G-Z$ 
does not contain any vertex labelled with~$\gamma$.
\end{enumerate}
\end{theorem}

\section{Necessity}

In this section we show that all labelled graphs $H$ such that the class of all $H$-expansion has the \EPP{} must have at least the properties stated in Theorem~\ref{thm:2consimple}.
We split the proof in several lemmas establishing gradually more properties of such $H$.

\begin{lemma}\label{outerlem}
Let $H$ be a labelled graph such that the labelled $H$-expansions have the 
Erd\H os-P\'osa property. Then 
there is an embedding of $H$  in the plane such that all 
its labelled vertices are on the boundary of the outer face.
\end{lemma}
\begin{proof}
First, we observe that we may assume $H$ to be planar.
Indeed, by Theorem~\ref{RSmetathm}, 
 non-planar graphs do not enjoy the (ordinary) Erd\H os-P\'osa property.
Then, if we label every vertex in any graph $G$ with all the labels of $H$, 
the labelled $H$-expansions do not have the \EPP~for the same reasons
as in the unlabelled case.

We thus assume that $H$ is a planar labelled graph that, however, does not have any embedding 
in the plane such that all its labelled vertices are on the boundary of the outer face.
Observe that, in particular, $H$ must have a component with that property.
Choose a minimum number $\ell$ such that there is an embedding of $H$ in the plane
in which the labelled vertices are contained in the union of $\ell$ face boundaries. 
By assumption, $\ell\geq 2$.

Let $R\in\mathbb N$ be sufficiently large, in a sense that will be made precise 
later in the proof.
Moreover, let $\Sigma$ be the alphabet containing all labels of $H$.
Consider a plane $\ell R\times \ell R$-grid, and pick $\ell$ mutually disjoint cycles
$C_1,\ldots C_\ell$, each  of length
at least $R$ (roughly $R/4\times R/4$ squares), so that each has distance at least $R/4$ from the outer face
and so that each two are at a distance of at least $R/4$
from each other. Let $G$ be the graph obtained by deleting the 
vertices in the interior of each $C_i$, and labelling every vertex in $\bigcup_{i=1}^\ell V(C_i)$
with all labels in $\Sigma$.

In what follows we see
 that every labelled $H$-expansion separates the interiors of the cycles $C_i$ from each other.
Then it will be easy to deduce that there are no two disjoint labelled $H$-expansions.
The fact that we choose $R$ large enough ensures that every hitting set has to be large (as its size grows with $R$).

Since $R$ is chosen to be large enough, $G$ contains a labelled $H$-expansion.
Indeed, for sufficiently large $R$ the graph $G$ contains an unlabelled $H$-expansion
such that every labelled vertex of $H$ maps to a branch set whose vertices of degree 
at least 3 are all contained in the same $C_i$ because there is an embedding of $H$ in the plane
in which the labelled vertices are contained in the union of $\ell$ face boundaries.
Such an unlabelled $H$-expansion is also a labelled $H$-expansion.

By increasing $R$, we can force the minimum size of a hitting set for labelled $H$-expansions to be arbitrarily large. 
Thus, to finish the proof it suffices to show that $G$ does not contain any two disjoint labelled $H$-expansions.

Let $H'$ be some labelled $H$-expansion in $G$. 
Denoting the interior faces
of the cycles $C_1,\ldots, C_\ell$ by $F_1,\ldots, F_\ell$, we see 
that $H'$ has a face $F'_i\supseteq F_i$ for each $i\in[\ell]$.
The faces $F'_1,\ldots, F'_\ell$ are pairwise distinct:
as the face boundaries of  $F'_1,\ldots, F'_\ell$ contain all the labelled vertices of $G$ in $H'$,
it follows from the
minimality of~$\ell$ that no two of these faces coincide.

Next, suppose there is a second labelled $H$-expansion $H''$ in $G$ that is disjoint
from $H'$. Again, the minimality of $\ell$ implies that $H''$ has a component 
that contains a vertex from $C_1$ as well as a vertex from $C_2$ (after relabelling $C_1,\ldots,C_\ell$).  
In particular, $H''$ contains a path $P$ that starts in a vertex of $C_1$ and ends in a vertex of $C_2$. 
Then, however, $P$ starts in $F_1'$ or in its boundary, and ends in $F'_2$ or in its boundary. 
As $F'_1\neq F_2'$ it follows that $P\subseteq H''$ meets $H'$, which shows that $H'$ and $H''$ 
are not disjoint.
\end{proof}

Recall that a labelled graph is simply-labelled if each labelled vertex only has one label and all labelled vertices have the same label.
We may use this notion for subgraphs of labelled graphs, too.

Let $H$ be a labelled planar graph $H$ that has an embedding in the plane in which 
all labelled vertices are on the boundary of the outer face. Define the 
\emph{label homogeneity} of $H$ as  the smallest integer $s$ such that 
for every sufficiently large integer $n$ there is a labelling of the vertices in the 
top row $P$ of an elementary $n$-wall $W$ such that $H$ is a labelled minor
of $W$ and such that there are $s$ simply-labelled subpaths of $P$ that cover all 
labelled vertices of $P$. 

\begin{figure}[ht]
\centering
\begin{tikzpicture}

\def\wallheight{12}
\def\brickheight{0.25}
\def\bh{\brickheight}

\tikzstyle{hvertex}=[thick,circle,inner sep=0.cm, minimum size=1.5mm, fill=white, draw=black]
\tikzstyle{avx}=[hvertex,fill=black]
\tikzstyle{bvx}=[hvertex,fill=dunkelgrau]
\tikzstyle{wed}=[draw,very thick,color=dunkelgrau]
%\tikzstyle{ed}=[draw,ultra thick, dotted]
\tikzstyle{ed}=[draw,line width=2pt]

\pgfmathtruncatemacro{\lastrow}{\wallheight}
\pgfmathtruncatemacro{\penultimaterow}{\wallheight-1}
\pgfmathtruncatemacro{\lastrowshift}{mod(\wallheight,2)}
\pgfmathtruncatemacro{\lastx}{2*\wallheight+1}

\draw[wed] (\brickheight,0) -- (2*\wallheight*\brickheight+\brickheight,0);
\foreach \i in {1,...,\penultimaterow}{
  \draw[wed] (0,\i*\brickheight) -- (2*\wallheight*\brickheight+\brickheight,\i*\brickheight);
}
\draw[wed] (\lastrowshift*\brickheight,\lastrow*\brickheight) to ++(2*\wallheight*\brickheight,0);

\foreach \j in {0,2,...,\penultimaterow}{
  \foreach \i in {0,...,\wallheight}{
    \draw[wed] (2*\i*\brickheight+\brickheight,\j*\brickheight) to ++(0,\brickheight);
  }
}
\foreach \j in {1,3,...,\penultimaterow}{
  \foreach \i in {0,...,\wallheight}{
    \draw[wed] (2*\i*\brickheight,\j*\brickheight) to ++(0,\brickheight);
  }
}

\foreach \i in {1,...,\lastx}{
  \coordinate (w\i w0) at (\i*\brickheight,0);
}
\foreach \j in {1,...,\penultimaterow}{
  \foreach \i in {0,...,\lastx}{
    \coordinate (w\i w\j) at (\i*\brickheight,\j*\brickheight);
  }
}
\foreach \i in {1,...,\lastx}{
  \coordinate (w\i w\lastrow) at (\i*\brickheight+\lastrowshift*\brickheight-\brickheight,\lastrow*\brickheight);
}

\draw[ed] (w11w\lastrow) to ++(4*\bh,0) to ++(0,-\bh) to ++(-4*\bh,0) to ++(0,\bh); 
\draw[ed] (w15w\lastrow) to (w19w\lastrow) to ++(0,-\bh) to ++(\bh,0) to ++(0,-\bh) to ++(-\bh,0)
          to ++(0,-\bh) to ++(-12*\bh,0) to ++(0,\bh) to ++(\bh,0) to ++(0,\bh) to ++(-\bh,0) to ++(0,\bh)
          to ++(2*\bh,0) to ++(0,-\bh) to ++(\bh,0) to ++(0,-\bh) to ++(8*\bh,0) to ++(0,\bh) to ++(-\bh,0)
          to ++(0,\bh);
\draw[ed] (w7w\lastrow) to ++(-4*\bh,0) to ++(0,-\bh) to ++(\bh,0) to ++(0,-\bh) to ++(-\bh,0)
          to ++(0,-\bh) to ++(\bh,0) to ++(0,-\bh) to ++(-\bh,0) to ++(0,-\bh)
          to ++(20*\bh,0) to ++(0,\bh) to ++(\bh,0) to ++(0,\bh) to ++(-\bh,0)
          to ++(0,\bh) to ++(\bh,0) to ++(0,\bh) to ++(-\bh,0) to ++(0,\bh) to ++(-2*\bh,0)
          to ++(0,-\bh) to ++(\bh,0) to ++(0,-\bh) to ++(-\bh,0)
          to ++(0,-\bh) to ++(\bh,0) to ++(0,-\bh) to ++(-\bh,0)
          to ++(-15*\bh,0) to ++(0,\bh) to ++(-\bh,0)
          to ++(0,\bh) to ++(\bh,0) to ++(0,\bh) to ++(-\bh,0)
          to ++(0,\bh);

\foreach \i in {3,5,7,9,11}{
  \node[avx] at (w\i w\lastrow) {};
}
\foreach \i in {13,15,17,19,21}{
  \node[bvx] at (w\i w\lastrow) {};
}

\node at (-0.9,1.5) {{\large $\mir$}};

\begin{scope}[shift={(-2,1.5)}]
\def\radius{0.3}
\draw[hedge] (0,0) circle (\radius);
\draw[hedge] (0,2*\radius) circle (\radius);
\draw[hedge] (0,-2*\radius) circle (\radius);
 
\node[avx]  at (0,\radius){};
\node[bvx]  at (0,-\radius){};
\node[avx]  at (\radius,0){};
\node[bvx]  at (-\radius,0){};
\node[avx]  at (\radius,2*\radius){};
\node[bvx]  at (-\radius,2*\radius){};
\node[avx]  at (\radius,-2*\radius){};
\node[bvx]  at (-\radius,-2*\radius){};
\node at (0,-4*\radius) {$H$};
\end{scope}

\end{tikzpicture}
\caption{A graph $H$ of label homogeneity~$2$}\label{homofig}
\end{figure}

\begin{lemma}\label{seplabelslem}
Let $H$ be a connected labelled graph that has an embedding in the plane in which 
all labelled vertices are on the boundary of the outer face.
If the labelled $H$-expansions have the 
Erd\H os-P\'osa property, then $H$ has label homogeneity at most~$2$.
\end{lemma}

In the proof we will consider two grids, each on a vertex set indexed by a set $[n]\times [n]$,
that is, on a vertex set $\{v_{ij}:(i,j)\in[n]\times [n]\}$. In both cases
we assume that the vertices are chosen in such a way that $v_{ij}$ is adjacent to $v_{i'j'}$
if and only if $i=i'$ and $|j-j'|=1$, or if $|i-i'|=1$ and $j=j'$.
The vertices $v_{jn}$ for $j\in[n]$ are the vertices of the \emph{top row} of the grid.

\begin{proof}[Proof of Lemma~\ref{seplabelslem}]
Suppose that $H$ has label homogeneity $\ell\geq 3$.
Then, there is a labelled $r \times r$ grid $G'$ for some sufficiently large $r$ such that $G'$ contains a labelled
$H$-expansion where all labelled vertices of $G'$ are contained in the top row, 
and such that there are $\ell$ disjoint simply-labelled subpaths $P'_1,\ldots,P'_\ell$ of the top row that cover
all its vertices.
Let $\{v_{ij}: i,j\in[r]\}$ be the vertex set of $G'$.

Suppose that $f$ is an Erd\H os-P\'osa function for labelled $H$-expansions.
We enlarge $G'$ to an $r'\times r'$-grid $G$ for $r'=rs=r\cdot 3(f(2)+1)$
with vertex set $\{w_{ij}:(i,j)\in [r']\times [r']\}$.
We say that $w_{ij}$ has \emph{pre-image} $v_{pq}$ 
if $i-(p-1)s\in [s]$ and $j-(q-1)s\in[s]$. 
We label a vertex $w_{jr'}$ in the top row of $G$ with label $\alpha$
if its pre-image $v_{pr}$ is labelled with $\alpha$ in $G'$.

Let $X$ be a set of at most $f(2)$  vertices in $G$. Let us convince ourselves that $G-X$
still contains a labelled $H$-expansion. For every $q\in[r]$, there is a $j\in[r']$
such that  none of the $(j-1)$th, the $j$th or the $(j+1)$th column meets $X$, and such that
$w_{j-1,r'}$, $w_{jr'}$ and $w_{j+1,r'}$ have $v_{qr}$ as pre-image; let $J$ be the set of these $j$, 
one for each $q\in [r]$. 
In a similar way, there are $r$ rows of $G$, with index set $I$,
that are disjoint from $X$. In particular, the union of the rows with index in $I$ and the columns with index in $J$
define a subgraph $F$ of $G-X$ that contains a subdivision of an $r\times r$-grid. 
Let $i_1$ be the largest integer in $I$.
We modify $F$  
by adding
for every $j\in J$ 
the path $w_{j-1,r'}w_{jr'}w_{j+1,r'}$ together with the three vertical paths from 
these vertices to 
$w_{j-1,i_1}$, $w_{j,i_1}$, and $w_{j+1,i_1}$ respectively. 
Call the obtained graph $F'$ and observe that the labelled grid $G'$ is a labelled minor of $F'$.
Due to Lemma~\ref{lem:transitivity}, $H$ is a labelled minor of $F'$.
Since $F'$ is disjoint from $X$, we see that no set of at most $f(2)$  vertices 
meets every labelled $H$-expansion. 

Therefore, 
$G$ must contain two disjoint labelled $H$-expansions, $H_1$ and $H_2$ say. 
By construction, the vertices of the top row of $G$ can be covered by $\ell$ disjoint simply-labelled
paths $Q_1,\ldots, Q_\ell$. By definition of the label homogeneity, each of the two 
the $H$-expansions  needs 
to contain at least one vertex from each of the paths $Q_1,\ldots, Q_\ell$.
Let $C_G$ be the boundary of the outer face of $G$. 

Starting with the plane graph $C_G\cup H_1\cup H_2$ we add 
a vertex $x$ drawn in the outer face of $C_G$ and make it 
adjacent to a vertex from each of $Q_1,Q_2,Q_3$. (Recall that $\ell\geq 3$.)
The resulting graph $K$ is planar. 
On the other hand, we see that $K$ has a $K_{3,3}$-minor by contracting 
each of $Q_1,Q_2,Q_3,H_1-(Q_1\cup Q_2 \cup Q_3),H_2-(Q_1\cup Q_2 \cup Q_3)$ to a single vertex, a contradiction.
This completes the proof.
\end{proof}

In Lemma~\ref{2homolem} we give a characterisation of the labelled graphs of label homogeneity at most~$2$.
To simplify its proof we use the following definition together with Lemma~\ref{conveniencelem}.
For a positive integer $h$ we define a 
graph $W(h)$ as follows. Start with an elementary $2h^2$-wall $W_1$, and let $n_1,\ldots,n_{2h^2}$
be the set of nails (in the order they appear in the top horizontal path).
We add to $W_1$ a set of $2h$ further vertices $a_1,\ldots,a_{h}$, $b_1,\ldots,b_h$,
and for each $i\in [h]$ we make   $a_i$ adjacent to each of $n_{(i-1)h+1},\ldots, n_{ih}$,
while we make $b_i$ adjacent to each of $n_{h^2+(i-1)h+1},\ldots,n_{h^2+ih}$.
The graph $W(h)$ is \emph{$(\alpha,\beta)$-labelled} if each vertex $a_1,\ldots,a_{h}$
is labelled with $\alpha$ and each of $b_1,\ldots,b_{h}$ is labelled with $\beta$.

\begin{figure}[ht]
\centering
\begin{tikzpicture}

\def\wallheight{18}
\def\brickheight{0.25}
\def\bh{\brickheight}

\tikzstyle{hvertex}=[thick,circle,inner sep=0.cm, minimum size=1.5mm, fill=white, draw=black]
\tikzstyle{avx}=[hvertex,fill=black]
\tikzstyle{bvx}=[hvertex,fill=dunkelgrau]
%\tikzstyle{ged}=[draw,line width=4pt,color=dunkelgrau]
\tikzstyle{wed}=[draw,very thick,color=dunkelgrau]
\tikzstyle{ed}=[draw,very thick]

\pgfmathtruncatemacro{\lastrow}{\wallheight}
\pgfmathtruncatemacro{\penultimaterow}{\wallheight-1}
\pgfmathtruncatemacro{\lastrowshift}{mod(\wallheight,2)}
\pgfmathtruncatemacro{\lastx}{2*\wallheight+1}

\clip[] (-0.5,3.1) rectangle (\wallheight*2*\brickheight+\brickheight+0.5,5.3);

\draw[wed] (\brickheight,0) -- (2*\wallheight*\brickheight+\brickheight,0);
\foreach \i in {1,...,\penultimaterow}{
  \draw[wed] (0,\i*\brickheight) -- (2*\wallheight*\brickheight+\brickheight,\i*\brickheight);
}
\draw[wed] (\lastrowshift*\brickheight,\lastrow*\brickheight) to ++(2*\wallheight*\brickheight,0);

\foreach \j in {0,2,...,\penultimaterow}{
  \foreach \i in {0,...,\wallheight}{
    \draw[wed] (2*\i*\brickheight+\brickheight,\j*\brickheight) to ++(0,\brickheight);
  }
}
\foreach \j in {1,3,...,\penultimaterow}{
  \foreach \i in {0,...,\wallheight}{
    \draw[wed] (2*\i*\brickheight,\j*\brickheight) to ++(0,\brickheight);
  }
}

\foreach \i in {1,...,\lastx}{
  \coordinate (w\i w0) at (\i*\brickheight,0);
}
\foreach \j in {1,...,\penultimaterow}{
  \foreach \i in {0,...,\lastx}{
    \coordinate (w\i w\j) at (\i*\brickheight,\j*\brickheight);
  }
}
\foreach \i in {1,...,\lastx}{
  \coordinate (w\i w\lastrow) at (\i*\brickheight+\lastrowshift*\brickheight-\brickheight,\lastrow*\brickheight);
}

\foreach \i in {1,2,3}{
  \node[avx,label=left:$a_\i$] (a\i) at (6*\i*\brickheight-3*\brickheight+\lastrowshift*\brickheight,\lastrow*\brickheight+2*\brickheight) {};
  \node[bvx,label=right:$b_\i$] (b\i) at (6*\i*\brickheight+15*\brickheight+\lastrowshift*\brickheight,\lastrow*\brickheight+2*\brickheight) {};
  \foreach \j in {1,2,3}{
    \draw[ed] (a\i) to (6*\i*\brickheight+2*\j*\brickheight+\lastrowshift*\brickheight-7*\brickheight,\lastrow*\brickheight);
    \draw[ed] (b\i) to (6*\i*\brickheight+2*\j*\brickheight+\lastrowshift*\brickheight+11*\brickheight,\lastrow*\brickheight);
  }
}

\end{tikzpicture}
\caption{The graph $W(3)$}\label{Whfig}
\end{figure}
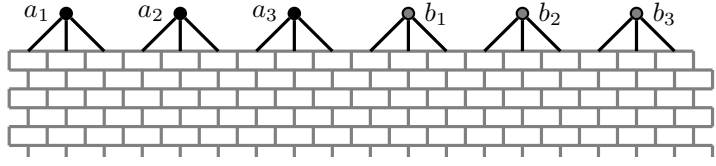

\begin{lemma}\label{conveniencelem}
Let $H$ be 
a labelled graph. 
Then $H$ has label homogeneity at most~$2$ if and only if 
$H$ is labelled with at most two labels, say $\alpha$ and $\beta$, and
there is an $h$
such that $H$ is a labelled minor of the $(\alpha,\beta)$-labelled graph $W(h)$.
\end{lemma}
\begin{proof}
One direction is easy: if $H$ has label homogeneity at most~$2$ then
it must be labelled with at most two labels, $\alpha$ and $\beta$, say, and there is a $t$
such that $H$ is a labelled minor of the elementary $2t$-wall $W'$ in which 
the first $t$ nails are labelled with $\alpha$ and the other $t$ nails with $\beta$.
As obviously $W'\mir W(t)$ it follows that also $H\mir W(t)$.

For the other direction, let $h$ be such that $H \mir W(h)$. 
Let $W_0$ be an elementary $(2h^2+2)$-wall, and let
$n_0,\ldots, n_{2h^2+1}$ be its nails (in the order as they appear in the top row). 
Label the nails $n_0,\ldots, n_{h^2}$
 with $\alpha$,
and label the other nails with $\beta$. 
We claim that $W(h)\mir W_0$.

To see this, denote by $Q$ the top row of $W_0$, and denote by 
$n^-_i$ the predecessor of $n_i$ on $Q$ for each $i$. We define branch sets $A_j$, $B_j$
for $j\in[h]$ as follows. Set $A_j=n^-_{(j-1)h+1}Qn_{jh-1}$ and 
$B_j=n^-_{(j-1)h+h^2+1}Qn_{jh+h^2-1}$. 
Taking in $W_0$ the  sets $A_j,B_j$ as branch sets, as well as all the vertices 
in $W_0-Q$ as singleton branch sets, 
we obtain a labelled $W(h)$-expansion, which means that $W(h)$, 
and thus also $H$, is a labelled minor of $W_0$. As the labels of $W_0$ 
can be covered by two simply-labelled subpaths of $Q$, it follows that $H$ has label homogeneity at most~$2$.
\end{proof}

\begin{lemma}\label{2homolem}
Let $H$ be a $2$-connected graph.
Then $H$ has label homogeneity at most~$2$
if and only if there is an embedding of $H$ in the plane such that the boundary $C$ of the outer 
face contains all labelled vertices, 
and there are two internally disjoint subpaths $P,Q\subseteq C$
that together cover all of~$V(C)$
and we can associate a label $\alpha$ with $P$ and a label $\beta$ with $Q$
such that
\begin{itemize}
	\item $P-Q$ is simply-labelled with $\alpha$ and 
	$Q-P$ is simply-labelled with $\beta$; and
	\item for all $v\in V(P\cap Q)$, we have $d_H(v)=2$ and $v$ is labelled with $\{\alpha,\beta\}$.
\end{itemize}
\end{lemma}
\begin{proof}
If $H$ has an embedding in the plane as stated above, there is an $h$ such that $H$ 
is a labelled minor of $W(h)$. By Lemma~\ref{conveniencelem}, 
it follows that $H$ has label homogeneity at most~$2$. 

If, on the other hand, $H$ has label homogeneity at most~$2$, then
there is an $h$ such that $H$ is a labelled minor of an elementary $2h$-wall $W$,
in which the first $h$ nails are labelled with $\alpha$ and the other $h$ nails are labelled 
with $\beta$. Let $(H',\pi)$ be a minimal labelled $H$-expansion in $W$.

If either $\alpha$ or $\beta$ are not used at vertices of $H$,
the statement of the lemma clearly holds,
as any labelled minor of $W$ has all its labels on the boundary of the same face.

We may, therefore, assume that some vertex in $H$ is labelled with $\alpha$ and some vertex
is labelled with $\beta$.
By contracting the branch sets of $H'$, we obtain a planar embedding of $H$.
Let $C'$ be the boundary of the outer face of $H'$, and let $C$ be the boundary of the outer face of the embedding of $H$.
Since $H$ is 2-connected, $C$ is a cycle
and since $H'$ is a minimal $H$-expansion,
$C'$ is also a cycle.
In fact, $C$ is obtained from $C'$ by contracting all branch sets.
As all labelled vertices of $W$ are contained in the top row and $C'$ is the boundary of the outer face, every labelled vertex of $H'$ must be on $C'$.
Hence $H$ has an embedding in the plane such that the boundary $C$ of the outer face contains all labelled vertices.

Consider the nails of $W$ ordered from left to right, say $n_1,\ldots,n_{2h}$, and let $n_i$ be the leftmost nail contained in $C'$.
Following $C'$ in clockwise fashion we obtain a sequence $(n_i=n_{i_1},n_{i_2},\ldots,n_{i_r})$ of all labelled vertices on $C'$.
Due to planarity, we have that $i_{j} < i_{j+1}$ for each $j\in [r-1]$.
By definition of $W$, there is some $j \in [r]$ such that $n_{i_j}$ is the rightmost nail labelled $\alpha$.

We observe that there are at most two vertices in $H$ that are labelled with 
$\{\alpha,\beta\}$
because every branch set of such a vertex must contain either $\{n_{i_j},n_{i_{j+1}}\}$
or $\{n_{i_r},n_{i_1}\}$.
Suppose $u\in V(H)$ is labelled with $\{\alpha,\beta\}$ and it contains both $n_{i_j}$ and $n_{i_{j+1}}$ (the argument for $n_{i_1}$ and $n_{i_r}$ is similar).
For a contradiction, assume that $d_H(u)\geq 3$.
Note that $T^\pi_u$ contains a vertex $x$ of degree at least $3$ on $C'$.
Observe that $x$ can be neither inside nor outside $n_{i_j}C'n_{i_{j+1}}$
as in both cases there is a leaf-to-leaf path in $T^\pi_u$ that either contains no vertex labelled $\alpha$ or 
no vertex labelled $\beta$.

Now it is not hard to construct the paths $P$ and $Q$ as in the statement.
\end{proof}

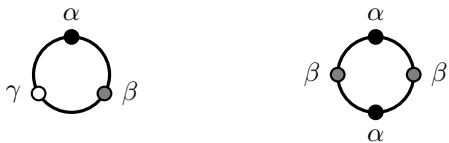
\begin{figure}[ht]
\centering
\begin{tikzpicture}
\def\radius{0.5}
\tikzstyle{hvertex}=[thick,circle,inner sep=0.cm, minimum size=1.8mm, fill=white, draw=black]
\tikzstyle{avx}=[hvertex,fill=black]
\tikzstyle{bvx}=[hvertex,fill=dunkelgrau]
\tikzstyle{cvx}=[hvertex,fill=white]

\draw[hedge] (0,0) circle (\radius);
\node[avx,label=above:$\alpha$] at (0,\radius){};
\node[bvx,label=right:$\beta$] at (-30:\radius){};
\node[cvx,label=left:$\gamma$] at (210:\radius){};

\begin{scope}[shift={(4,0)}]
\draw[hedge] (0,0) circle (\radius);
\node[avx,label=above:$\alpha$] at (0,\radius){};
\node[avx,label=below:$\alpha$] at (0,-\radius){};
\node[bvx,label=right:$\beta$] at (\radius,0){};
\node[bvx,label=left:$\beta$] at (-\radius,0){};
\end{scope}

\end{tikzpicture}
\caption{Two graphs of label homogeneity larger than~$2$}\label{homo2fig}
\end{figure}

With Lemma~\ref{2homolem} we can see that neither of the graphs in Figure~\ref{homo2fig}
has label homogeneity at most~$2$, which in light of the other results in this
section means that the expansions of neither of the graphs have the Erd\H os-P\'osa property.

\section{\EPP\ for $2$-connected $H$}

In this section, we prove a slightly stronger version of our main result, Theorem~\ref{thm:2consimple}.

\begin{theorem}\label{thm:2con}
Let $H$ be a labelled $2$-connected graph.
Then the labelled $H$-expansions 
have the \EPP{} if and only if there is an embedding of $H$ in the plane such that the boundary $C$ of the outer 
face contains all labelled vertices, and there are two internally disjoint subpaths $P,Q\subseteq C$
that cover all of~$V(C)$
and we can associate a label $\alpha$ with $P$ and a label $\beta$ with $Q$
such that
\begin{itemize}
	\item $P-Q$ is simply-labelled with $\alpha$ and 
	$Q-P$ is simply-labelled with $\beta$; and
	\item for all $v\in V(P\cap Q)$, we have $d_H(v)=2$ and $v$ is labelled with $\{\alpha,\beta\}$.
\end{itemize}
\end{theorem}

Note that Theorem~\ref{thm:2con} clearly implies Theorem~\ref{thm:2consimple}.
The proof closely follows the different outcomes of Theorem~\ref{thm:simpleflatwall}.

For two labels $\alpha,\beta$ we write $K_n^{\alpha,\beta}$
for  the complete graph on $n$ vertices in which every vertex is labelled with~$\{\alpha,\beta\}$.

\begin{lemma}\label{lem:rootedKn}
Let $t\geq 3$, and let $\alpha,\beta$ be labels.
Then every $(\alpha,\beta)$-thoroughly labelled $K_{6t^2-5t}$-expansion contains a 
labelled $K_t^{\alpha,\beta}$-expansion.
\end{lemma}
\begin{proof}
Let $K$ be the complete graph on the vertex set
\[
c^1,\ldots, c^t,\, v^{ij}_1,\ldots,v^{ij}_6\text{ for all distinct }i,j\in [t]
\]
of $6t^2-5t$ distinct vertices. Let $(X,\pi)$ be a $(\alpha,\beta)$-thoroughly labelled $K$-expansion. 

Consider arbitrary distinct indices $i,j\in [t]$.
By definition, there is a cycle~$C$ in 
\[
\pi(v^{ij}_1)\cup\pi(v^{ij}_1v^{ij}_2)\cup 
\pi(v^{ij}_2)\cup\pi(v^{ij}_2v^{ij}_3)\cup
\pi(v^{ij}_3)\cup\pi(v^{ij}_3v^{ij}_1)
\]
that contains a vertex with label~$\alpha$. 
By renaming the vertices $v^{ij}_1,v^{ij}_2,v^{ij}_3$
if necessary we may assume that there is a $\pi(v^{ij}_1)$--$\pi(v^{ij}_3)$~path $P_1$
in $C$ that contains a vertex in $\pi(v^{ij}_2)$ with label~$\alpha$. With an analogous 
argument, we may assume that there is a    $\pi(v^{ij}_4)$--$\pi(v^{ij}_6)$~path $P_2$
contained in 
\[
\pi(v^{ij}_4)\cup\pi(v^{ij}_4v^{ij}_5)\cup 
\pi(v^{ij}_5)\cup\pi(v^{ij}_5v^{ij}_6)\cup
\pi(v^{ij}_6)
\]
that contains a vertex in $\pi(v^{ij}_5)$ of label~$\beta$.
Using an edge in $\pi(v^{ij}_3v^{ij}_4)$,
as well as an edge in $\pi(c^iv^{ij}_1)$,
we can find a $\pi(c^i)$--$\pi(v^{ij}_6)$~path $Q^{ij}$ that contains both a 
vertex with label~$\alpha$ and a vertex with label~$\beta$ in its interior
and that itself is contained in the induced graph on 
$\pi(c^i)\cup\bigcup_{\ell=1}^6 \pi(v^{ij}_\ell)$.

Having constructed all such paths $Q^{ij}$, let $ab\in\pi(v^{ij}_6v^{ji}_6)$
with $a\in\pi(v^{ij}_6)$. Set $\pi'(ij)=ab$, and 
define a tree $T'_i$ by taking the union of all paths $Q^{ij}$, 
$a$--$Q^{ij}$~paths in $\pi(v^{ij}_6)$ together with a minimal subtree of $\pi(c^i)$
so as to result in a tree. Set $\pi'(i)=V(T'_i)$, and observe that $\pi'$ 
defines a $K_t$-expansion $Y$ that is contained in $X$. In $Y$ the trees $T^{\pi'}_i$ 
(recall the definition of a labelled expansion)
consist of $T'_i$ together with all edges $\pi'(ij)$. Every leaf-to-leaf path
in $T^{\pi'}_i$ passes through $\pi(c^i)$ and then contains $Q^{ij}$ and $Q^{ik}$
for two $j,k$. Consequently, every leaf-to-leaf path contains a vertex with label~$\alpha$
and a vertex  with label~$\beta$ that lies in $\pi'(i)$.
Therefore, $(Y,\pi')$ is a labelled $K_t^{\alpha,\beta}$-expansion.
\end{proof}

\begin{lemma}\label{caseilem}
Let $H$ be an $(\alpha,\beta)$-labelled graph. 
For every $k$ (and $H$), there is a $t$ such that every $(\alpha,\beta)$-thoroughly labelled $K_t$-expansion contains $k$ disjoint 
labelled $H$-expansions.
\end{lemma}
\begin{proof}
Set $h=k|V(H)|$, and observe that there are $k$ disjoint labelled minors of $H$ in $K_h^{\alpha,\beta}$.
Set $t=6h^2-5h$, and apply Lemma~\ref{lem:rootedKn} in order to find $K_h^{\alpha,\beta}$
as a labelled minor in any $(\alpha,\beta)$-thoroughly labelled $K_t$-expansion.
\end{proof}

\begin{lemma}\label{serieslem}
Suppose $t\geq 9r$.
If $W$ is a $(t+1)$-wall with an $\alpha$-clean linkage of size $2r$, there is a $t$-subwall $W'$ of $W$ with an $\alpha$-clean linkage of size $r$ that is in series.
Moreover, if $W$ has an $(\alpha,\beta)$-clean pair of linkages of size $2r$, there is a $t$-subwall $W'$ of $W$ with an $(\alpha,\beta)$-clean pair of linkages that are both in series and of size $r$.
\end{lemma}
\begin{proof}
We prove the second statement since the first one follows in the same way.
Let $(\cP,\cQ)$ be an $(\alpha,\beta)$-clean pair of linkages of size $2r$, and let $R$ be the top row of $W$.
For each nail $u$, let $S_u$ be the path contained in $W$ from $u$ to the upper right corner
 and then to the lower right corner of its brick in $W$; see Figure~\ref{seriesfig}. 
Let $\cP=\{P_1,\ldots,P_{2r}\}$ be the paths in $\cP$ and denote their left endvertices
by $p_1,\ldots,p_{2r}$,
and the corresponding right endvertices by $p'_1,\ldots,p'_{2r}$.
Assume $p_1<\ldots<p_{2r}$ where the ordering is from left to right in the top row $R$ of $W$.
Moreover, let $\cQ = \{Q_1,\ldots,Q_{2r}\}$ with left endvertices $q_1,\ldots,q_{2r}$
and right endvertices $q'_1,\ldots,q'_{2r}$ be ordered in the same way.

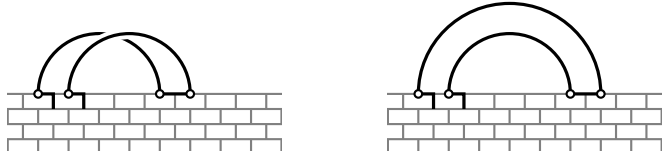
\begin{figure}[ht]
\centering
\begin{tikzpicture}

\tikzstyle{oben}=[line width=1.5pt,double distance=1.2pt,draw=white,double=black]
\tikzstyle{tinyvx}=[thick,circle,inner sep=0.cm, minimum size=1mm, fill=white, draw=black]
\tikzstyle{marked}=[line width=2pt,color=dunkelgrau]
\tikzstyle{wed}=[thick,color=dunkelgrau]

\def\wallheight{10}
\def\brickheight{0.2}

\pgfmathtruncatemacro{\lastrow}{\wallheight}
\pgfmathtruncatemacro{\penultimaterow}{\wallheight-1}
\pgfmathtruncatemacro{\lastrowshift}{mod(\wallheight,2)}
\pgfmathtruncatemacro{\lastx}{2*\wallheight+1}

\pgfmathsetmacro{\topy}{\wallheight*\brickheight}

\begin{scope}
\clip[] (\brickheight,\topy-0.75) rectangle (2*\wallheight*\brickheight-\brickheight,\topy+1);

\draw[oben] (3*\brickheight,\topy) arc (180:0:4*\brickheight);
\draw[oben] (5*\brickheight,\topy) arc (180:0:4*\brickheight);

\draw[wed] (\brickheight,0) -- (2*\wallheight*\brickheight+\brickheight,0);
\foreach \i in {1,...,\penultimaterow}{
  \draw[wed] (0,\i*\brickheight) -- (2*\wallheight*\brickheight+\brickheight,\i*\brickheight);
}
\draw[wed] (\lastrowshift*\brickheight,\lastrow*\brickheight) to ++(2*\wallheight*\brickheight,0);

\foreach \j in {0,2,...,\penultimaterow}{
  \foreach \i in {0,...,\wallheight}{
    \draw[wed] (2*\i*\brickheight+\brickheight,\j*\brickheight) to ++(0,\brickheight);
  }
}
\foreach \j in {1,3,...,\penultimaterow}{
  \foreach \i in {0,...,\wallheight}{
    \draw[wed] (2*\i*\brickheight,\j*\brickheight) to ++(0,\brickheight);
  }
}

\draw[hedge] (3*\brickheight,\topy) -- ++(\brickheight,0) -- ++(0,-\brickheight);
\draw[hedge] (5*\brickheight,\topy) -- ++(\brickheight,0) -- ++(0,-\brickheight);
\draw[hedge] (11*\brickheight,\topy) -- ++(2*\brickheight,0);

\node[tinyvx] at (3*\brickheight,\topy){};
\node[tinyvx] at (5*\brickheight,\topy){};
\node[tinyvx] at (11*\brickheight,\topy){};
\node[tinyvx] at (13*\brickheight,\topy){};
\end{scope}

\begin{scope}[shift={(5,0)}]
\clip[] (\brickheight,\topy-0.75) rectangle (2*\wallheight*\brickheight-\brickheight,\topy+1.5);

\draw[oben] (3*\brickheight,\topy) arc (180:0:6*\brickheight);
\draw[oben] (5*\brickheight,\topy) arc (180:0:4*\brickheight);

\draw[wed] (\brickheight,0) -- (2*\wallheight*\brickheight+\brickheight,0);
\foreach \i in {1,...,\penultimaterow}{
  \draw[wed] (0,\i*\brickheight) -- (2*\wallheight*\brickheight+\brickheight,\i*\brickheight);
}
\draw[wed] (\lastrowshift*\brickheight,\lastrow*\brickheight) to ++(2*\wallheight*\brickheight,0);

\foreach \j in {0,2,...,\penultimaterow}{
  \foreach \i in {0,...,\wallheight}{
    \draw[wed] (2*\i*\brickheight+\brickheight,\j*\brickheight) to ++(0,\brickheight);
  }
}
\foreach \j in {1,3,...,\penultimaterow}{
  \foreach \i in {0,...,\wallheight}{
    \draw[wed] (2*\i*\brickheight,\j*\brickheight) to ++(0,\brickheight);
  }
}

\draw[hedge] (3*\brickheight,\topy) -- ++(\brickheight,0) -- ++(0,-\brickheight);
\draw[hedge] (5*\brickheight,\topy) -- ++(\brickheight,0) -- ++(0,-\brickheight);
\draw[hedge] (13*\brickheight,\topy) -- ++(2*\brickheight,0);

\node[tinyvx] at (3*\brickheight,\topy){};
\node[tinyvx] at (5*\brickheight,\topy){};
\node[tinyvx] at (15*\brickheight,\topy){};
\node[tinyvx] at (13*\brickheight,\topy){};
\end{scope}

\end{tikzpicture}
\caption{How to turn a crossing or nested linkage into one that is in series}\label{seriesfig}
\end{figure}

Consider the paths 
\[
P'_i = S_{p_{2i-1}}p_{2i-1}P_{2i-1}p'_{2i-1}Rp'_{2i}P_{2i}p_{2i}S_{p_{2i}} 
\] 
for each $i\in [r]$
if $\cP$ is crossing or nested (see Figure~\ref{seriesfig}),
otherwise let $P'_i = S_{p_{2i}}p_{2i}P_{2i}p'_{2i}S_{p_{2i}}$.
We define $Q'_i$ analogously. Note that for each $i,j\in[r]$ the paths $P'_i$ and $Q'_j$
are pairwise disjoint --- this is due to the last condition in the definition 
of a pair of clean $(\alpha,\beta)$-linkages. 
Let $W'$ be the $t$-subwall obtained from $W$ by deleting the top row and leftmost column, 
let $\cP'=\{P'_i:i\in [r]\}$, and $\cQ'=\{Q'_i:i\in [r]\}$.
Note that the pair $(\cP',\cQ')$ is $(\alpha,\beta)$-clean for $W'$, 
which completes the proof.
\end{proof}

\begin{lemma}\label{linkredlem}
Let $r\geq 4t$, and let $W$ be a $100r$-wall that is either thoroughly
$(\alpha,\beta)$-labelled, or that is thoroughly $\alpha$-labelled and has a $\beta$-clean linkage
of size $r$. Then $W$ contains a $100t$-subwall $W'$ that has an $(\alpha,\beta)$-clean pair of linkages
of size $t$ such that both linkages are in series.
\end{lemma}
\begin{proof}
First, by Lemma~\ref{serieslem}, if $W$ has a $\beta$-clean linkage of size $2r$
(rather than being thoroughly $(\alpha,\beta)$-labelled)
then it also has such a linkage of size $r\geq 2t$ that is in series --- at 
the price of reducing the size of the wall by~$1$.

Pick a vertical path $P$ of $W$ such that each of the two components $W_1,W_2$ of $W-P$ contains 
at least $49r$ of the vertical paths of $W$. We may assume that if $W$ has a $\beta$-clean linkage (which then 
is in series), then at least half of the paths of the linkage have both endvertices in $W_1$.
That is, $W_1$ has a $\beta$-clean linkage of size $t$.
Also, for each $i\in [2]$, let $W_i'$ be obtained by $W_i$ by deleting the first two horizontal paths.

Let $B_1,\ldots, B_t$ be a choice of $r$ (vertex-)disjoint bricks from the top row of $W_2$.
Let $Q_3$ be the third horizontal path of $W_2$ from the top; that is, the top path of $W_2'$. 
There are $2t$
disjoint $Q_3$--$\bigcup_{i=1}^tB_i$ paths $R_1,\ldots, R_{2t}$ such that 
$R_{2i-1}$ and $R_{2i}$ end in $B_i$ for each $i\in[t]$. 
Since each brick of $W_2$ contains a vertex labelled with $\alpha$ 
as $W$ is thoroughly $\alpha$-labelled, for each $i$, one of the two paths in $B_i$ between 
the endvertices of $R_{2i-1}$ and $R_{2i}$ contains a vertex of label~$\alpha$. 
Denote this subpath by $S_i$. 
Hence $(R_{2i-1}\cup S_i\cup R_{2i})_{i\in [t]}$
is an $\alpha$-clean linkage in series of $W_2'$ of size $t$. 

If $W$ is thoroughly $(\alpha,\beta)$-labelled we repeat this procedure in $W_1$ with 
the label $\beta$. 
If $W$ has a $\beta$-clean linkage,
then, by prolonging the linkage through
the wall to $W_1'$, we obtain a $\beta$-clean linkage of $W'_1$ of size at least $t$. 
In both cases, by 
using the horizontal paths that link $W'_1$ and $W_2'$ we find a $100t$-wall $W'$ as a subwall with an 
$(\alpha,\beta)$-clean pair of linkages of size at least $t$. Moreover, the linkages are in series. 
\end{proof}

Note that if $W$ is a wall with an 
$(\alpha,\beta)$-clean pair of linkages which is in series, the union of these two linkages is itself a linkage that is in series.
This follows from the definition of a $(\alpha,\beta)$-clean pair.
Hence, we may simply say that an 
$(\alpha,\beta)$-clean pair of linkages is in series if both linkages are in series.

\begin{lemma}\label{caseiiiclem} 
Let $H$ be an $(\alpha,\beta)$-labelled graph that has an embedding in the plane 
such that all labelled vertices lie in the boundary of the outer face, 
and assume $H$ to have label homogeneity at most~$2$. 
For every $k$ there is a $t$
such that the following holds: whenever $W$ is a wall of size at least $20t$ that has an in series $(\alpha,\beta)$-clean pair of linkages of size $t$,
then $W$ together with the linkage contains $k$ disjoint labelled $H$-expansions.
\end{lemma}
\begin{proof}
For a positive integer $t'$, let $U_{t'}$ be a labelled graph consisting
of an elementary $2t'$-wall 
where the first $t'$ nails are labelled $\alpha$ and the remaining ones are labelled $\beta$.
As $H$ has label homogeneity~$2$, there is a $t'$ such that $H$ is a labelled minor of $U_{t'}$. 
We now fix such a $t'$ and simply 
write $U$ instead of $U_{t'}$.
We will find $k$ disjoint labelled $U$-expansions, that then contain $k$ disjoint labelled
$H$-expansions.

We set $t = 5kt'$.
Let $(\cP,\cQ)$ be an $(\alpha,\beta)$-clean pair of linkages of the wall $W$ of size $t$
that is in series.
As $\cP$ and $\cQ$ are in series, all paths in $\cP$ connect to the top row of $W$ left of all paths in $\cQ$ or vice versa.
In particular, $W$, which has size at least $20t=100kt'$, together with $\cP$ and $\cQ$ 
contains as a labelled minor a $10kt'$-wall $W'$ in which the 
first $5kt'$ nails are labelled $\alpha$ and the 
remaining $5kt'$ nails are labelled $\beta$.
We claim that
\begin{equation}\label{wallcarving}
\emtext{
$W'$ contains $k$ disjoint labelled $U$-expansions.
}
\end{equation} 
The claim is proved by induction on $k$.
For $k=1$, \eqref{wallcarving} holds as $U$ is a labelled minor of $W'$.
Suppose now that $k>1$.
Let $W''$ be a subwall of $W'$ of size $10(k-1)t'$ that
 contains exactly $5(k-1)t'$  $\alpha$-labelled and
and exactly $5(k-1)t'$ $\beta$-labelled nails of $W'$ such that the horizontal
and vertical paths of $W'$ that $W''$ meets are contiguous.
By induction, $W''$ contains $k-1$ labelled $U$-expansions. 
The graph  $\tilde{W}=W'-W''$ contains the leftmost and rightmost $5t'-1$ vertical paths of $W'$,
as well as the $5t'$ bottommost horizontal paths.
Then,  $\tilde{W}$ contains $U$ as a labelled minor.
In total, we have found $k$ disjoint labelled $U$-expansions.
This proves~\eqref{wallcarving} and the lemma.
\end{proof}

\begin{lemma}\label{caseiiilem}
Let $H$ be an $(\alpha,\beta)$-labelled graph that has an embedding in the plane 
such that all labelled vertices lie in the boundary of the outer face, 
and assume $H$ to have label homogeneity at most~$2$. 
For every $k$ there is a $t$
such that the following holds: if a graph $G$ consists
of a wall $W$ of size at least $100t$ such that
	\begin{enumerate}[label=\rm (\alph*)]
		\item $W$ is $(\alpha,\beta)$-thoroughly labelled,
		\item for some $\gamma\in \{\alpha,\beta\}$, 
		the wall $W$ is $\gamma$-thoroughly labelled and
		has an $(\{\alpha,\beta\}\sm \gamma)$-clean linkage of size $t$, or
		\item $W$ has a pair of $(\alpha,\beta)$-clean linkages of size $t$,
	\end{enumerate}
then $G$ contains $k$ disjoint labelled $H$-expansions.
\end{lemma}
\begin{proof}
For a given $k$, let $s$ be as the $t$ in the statement of
 Lemma~\ref{caseiiiclem}, and set $t=4s$. Then, with Lemma~\ref{linkredlem},
we may assume that $W$ has size $100s$ and comes with a 
 $(\alpha,\beta)$-clean pair of linkages of size $s$ that are in series. Lemma~\ref{caseiiiclem}
now yields the $k$ disjoint labelled $H$-expansions.
\end{proof}

We can now prove our main result.

\begin{proof}[Proof of Theorem~\ref{thm:2con}]
Necessity follows from Lemmas~\ref{outerlem},~\ref{seplabelslem} and~\ref{2homolem}.

It remains to prove sufficiency.
For this, let $H$ be a $2$-connected $(\alpha,\beta)$-labelled graph that has an embedding as in the statement.
To proceed to the difficult case, we assume that $H$ contains vertices of both labels, $\alpha$ and $\beta$.
(If not, set $\alpha=\beta$.)

Suppose the theorem is false. Then there is a largest $k<\infty$ such that
there are values $f(2),\ldots,f(k-1)$ such that  
for all $k'<k$ every graph $G$ either contains $k'$ disjoint labelled $H$-expansions
or a vertex set $X$ of size $|X|\leq f(k')$ that meets every $H$-expansion.

Fix numbers $t_1\gg t_2\gg t_3 \gg k$, where we make precise
what that means below.
Moreover, choose $f(k)$ 
such that $t_1 \le \min\{f(k) - 2f(k-1),f(k)/3\}$
and complete $f$ to a function $f:\mathbb N\to \mathbb N$.

By the choice of $k$ we may pick a minimal counterexample $(G,k)$ 
to $f$ being an \EP~function for the family of labelled $H$-expansions. 
Let $\cT$ be the tangle as defined in Lemma~\ref{largetanglelem} with $t_1$ playing the role of $t$ which, by Lemma~\ref{largetanglelem}, has size at least $t_1$.

We assume $t_1$ to be chosen large enough 
such that Theorem~\ref{lemma: tangle wall} yields a $t_2$-wall~$W_1$ whose induced tangle
$\mathcal T_{W_1}$ is a truncation of $\mathcal T$.
Next, we assume $t_2$ to be
large enough such that $t_2$ and $t_3$ can play the roles of $t'$ and $t$
in Theorem~\ref{thm:simpleflatwall}.

We now go through the different outcomes of Theorem~\ref{thm:simpleflatwall}.
For outcome~\ref{case1}, we apply Lemma~\ref{caseilem}, 
where we choose $t_3$  large enough to yield $k$ disjoint $H$-expansions. 
For outcome~\ref{case2}, we apply Lemma~\ref{caseiiilem}, where again we assume 
that~$t_3$ is large enough to ensure $k$ disjoint $H$-expansions.

Finally, we observe that the outcome~\ref{case3} may not occur.
Indeed, recall that~\ref{case3} yields a label $\gamma\in\{\alpha,\beta\}$ and
a set $Z$ such that 
$|Z|< t_2$ and the unique $\cT_{W_1}$-large block of $G-Z$ 
does not contain any vertex labelled with $\gamma$.
Since $H$ is 2-connected by assumption, any labelled $H$-expansion in $G-Z$ is edge-disjoint from the $\cT_{W_1}$-large block of $G-Z$.
As $|Z| < t_2 \le f(k)$
and $(G,k)$ is a counterexample, however, $G-Z$ must contain some labelled $H$-expansion.
Thus, there is a separation $(A,B)$ of order at most $|Z|+1<t_1 $  such that $B$
contains the unique $\cT_{W_1}$-large block of $G-Z$. Hence, $(A,B)\in \cT$, but $A$ contains a labelled $H$-expansion (hence $(B,A)\in \cT$),
which is a contradiction to~\ref{T2}.
This completes the proof.
\end{proof}

\section{Zero cycles and zero $H$-expansions}\label{zerosec}

A good number of Erd\H os-P\'osa type results concern paths and cycles, 
and in many of these additional parity or modularity constraints are imposed on the
length of the paths or cycles. The first of these is certainly the result
of Dejter and Neumann-Lara (see~\cite{DNL87}) that \emph{even} cycles have the Erd\H os-P\'osa
property. Thomassen~\cite{Tho88} extended their result to encompass, in particular, 
all cycles whose length is divisible by a fixed integer. 
In this section, we outline how our main theorem can be extended to include such modularity
constraints, too. 

For paths, 
Thomassen's modularity constraints were further generalised  by 
 endowing the edges of the host graph with weights, or labels, from a group. 
Two different approaches were pursued:
Chudnovsky et al.~\cite{CGGGLS06} considered directed group labellings, 
while 
Wollan~\cite{Wol10} treated
 undirected group labellings. We focus here on undirected labellings, 
but also mention how directed labellings can be accomodated with similar arguments. 
Moreover, we prefer to talk about \emph{group weights} so that they do not get confused with 
the vertex labels we have treated so far.

% we assume the group to be abelian here
% why? if the labelling is undirected and the group non-abelian
% then the weight of a path might equal the neutral element 
% when computed starting with one endvertex but not when we
% start with the other endvertex
% HOWEVER: if the labelling is directed this cannot occur
% thus, everything here should also go through for directed
% labellings with a non-abelian group
Let $\Gamma$ be an abelian group. 
Assume that the edges of a graph $G$ are endowed with directed or undirected weights $\gamma$ 
from $\Gamma$. For undirected weights, this simply means that a function $\gamma:E(G)\to \Gamma$ is 
fixed. For directed weights, on the other hand, we first pick an arbitrary 
 reference orientation $\vec G$ of $G$.
To keep notation consistent with the undirected case, we interpret $uv$ as an edge pointing 
from $u$ to $v$, while $vu$ denotes the inverse edge that points from~$v$ to~$u$. 
We now define directed weights as a 
function $\gamma$ on the edges of $\vec G$ as well as their inverses 
such  that $\gamma(uv)=-\gamma(vu)\in\Gamma$ for every edge $uv$ of $\vec G$. 
(Note that~$vu$ will not be an edge of $\vec G$ if $uv\in E(\vec G)$.)

Let $P=v_0v_1\ldots v_\ell$ be a path in $G$, where in the directed setting we implicitly fix~$v_0$
as first vertex.
The \emph{weight} of $P$ is defined as 
\[
\gamma(P)=\gamma(v_0v_{1})+\ldots+\gamma(v_{\ell-1}v_\ell).
\]
Note that, in the directed setting we sum up the weights of the edges 
directed from $v_i$ to $v_{i+1}$. Still in the directed setting, we furthermore observe that 
while normally it makes a difference in which direction we sum the edges of $P$, 
it does not if the weight of $P$ is~$0$.
Indeed, denote by $P'$ the path $P$, only with $v_\ell$ as first vertex and assume that $\gamma(P)=0$.
Then
\[
0=-0=-\gamma(v_0v_{1})-\ldots-\gamma(v_{\ell-1}v_\ell)=\gamma(v_{\ell}v_{\ell-1})+\ldots+\gamma(v_1v_0)
=\gamma(P').
\]
A path whose weight is~0 is a \emph{zero path}.

We now formulate our extension of Theorem~\ref{thm:2consimple} to 
graphs with undirected group weights. For this, we say that 
an $H$-expansion in $G$ is a \emph{zero} $H$-expansion if the 
sum of the weights of all the edges in the expansion is~$0$. 
We later briefly discuss directed group weights. 

\begin{theorem}\label{zerothm}
Let $\Gamma$ be a finite abelian group.
Let $H$ be a labelled $2$-connected graph such that each vertex carries at most one label.
Then the zero labelled $H$-expansions have the Erd\H os-P\'osa property if  
there is an embedding of $H$ in the plane such that the boundary $C$ of the outer 
face contains all labelled vertices, and there are two simply-labelled subpaths $P,Q\subseteq C$
that cover all of~$V(C)$.
\end{theorem}

Kakimura and Kawarabayashi~\cite{KK12} proved that the $A$-cycles whose length is divisible by~$p$
have the Erd\H os-P\'osa property. From Theorem~\ref{zerothm} two different generalisations can be deduced.

\begin{corollary}
Let $p$ be any positive integer. Then the $A$--$B$-cycles whose length is divisible by~$p$
have the Erd\H os-P\'osa property. 
\end{corollary}

For the second generalisation, given a vertex set $A$ and an integer $\ell$, call a cycle
an \emph{$\ell\cdot A$-cycle} if it contains at least $\ell$ vertices from $A$.
In this sense, the result of Kakimura and Kawarabayashi is about $1\cdot A$-cycles. 

\begin{corollary}
Let $\ell$ be any positive integer and let $\Gamma$ be a finite abelian group. Then the zero $\ell\cdot A$--cycles
have the Erd\H os-P\'osa property. 
\end{corollary}

We now indicate how the proof of Theorem~\ref{thm:2consimple} has to be adapted. 
We only provide a sketch of the proof as it simply amounts to combining 
techniques of Thomassen~\cite{Tho88} with the proof of our main theorem.

We formulate three lemmas. 
The  first lemma as well as its proof is a straightforward generalisation of an argument originally
made by Thomassen~\cite{Tho88}.
\begin{lemma}\label{zeropathlem}
Let $\Gamma$ be a finite abelian group, and let~$r$ be a positive integer.
If~$P$ is a path with (directed or undirected) $\Gamma$-weights and if
and $U\subseteq V(P)$ has size at least $r|\Gamma|$,
then there is a $W\subseteq U$ of size at least $r$ such that any subpath of~$P$ between two
vertices in~$W$ has zero weight. 
\end{lemma}
%Note that the lemma becomes false if the group is allowed to be infinite. With undirected weights
%this is particularly easy to see -- it suffices to endow every edge with the weight~$1$ from $\mathbb Z$.
\begin{proof}
Let $u_0$ be the first vertex of $U$ along $P$, and for every $u\in V(P)$, denote by $\gamma_u$
the weight of the path from $u_0$ to $u$. As there are only $|\Gamma|$ many possible different values
for $\gamma_u$ but $|U\sm\{u_0\}|>(r-1)|\Gamma|$ it follows that there is a subset~$W$ of $U\sm\{u_0\}$
of size $r$ such that all the values $\gamma_w$, $w\in W$, coincide. 

Consider $v,w\in W$ such that $v$ lies on the subpath between $u_0$ and $w$. Then 
\[
\gamma(u_0Pv)+ \gamma(vPw)=\gamma(u_0Pw)=\gamma_w=\gamma_v=\gamma(u_0Pv),
\] 
which implies that $vPw$ is a zero path.
\end{proof}

The next lemma is also a straightforward generalisation of an observation of Thomassen~\cite{Tho88}.
A wall with (directed or undirected) group weights is a \emph{zero wall}
if all its subdivided edges are zero paths. 
\begin{lemma}\label{zerowalllem}
Let $\Gamma$ be a finite abelian group.
Then, for every positive integer $r$, 
there is an integer $s$ such that whenever $W$ is 
a wall of size $s$ with (directed or undirected) $\Gamma$-weights then $W$ has a zero subwall $W'$
of size $r$.
\end{lemma}
\begin{proof}[Proof sketch]
We begin by choosing many vertical paths pairwise sufficiently far apart. We apply Lemma~\ref{zeropathlem}
to each of the paths and then use the horizontal as well as the other vertical paths to find a still 
large subwall in which all subdivided edges of the vertical paths are zero paths. We repeat this process
(in the subwall) for the horizontal paths. 
\end{proof}

%\begin{lemma}
%For every positive integer $r$ there is an integer $s$ such that every thoroughly $(\alpha,\beta)$-labelled
%wall $W$ contains a subwall $W'$ of size $s$ such that every subdivided edge of $W'$ contains both labels.
%\end{lemma}

\begin{lemma}
Let $\Gamma$ be a finite abelian group.
For every positive integer $r$, there is an integer $s$ such that every 
$100s$-wall $W$ with (directed or undirected) $\Gamma$-weights
that has an $(\alpha,\beta)$-clean pair of linkages of size $s$
that are both in series, 
 has a zero $100r$-subwall $W'$ with an $(\alpha,\beta)$-clean pair of linkages of size $r$ 
such that both linkages are in series and zero.
\end{lemma}
\begin{proof}[Proof sketch]
We first apply Lemma~\ref{zerowalllem} to a large subwall that leaves out enough rows at the top
so that we can reattach the pair of linkages to the resulting zero wall. Then we concatenate 
the paths in each of the linkages, with parts of the top row, to a very long path to which 
we
apply Lemma~\ref{zeropathlem}. We regain an $(\alpha,\beta)$-clean pair of linkages
by splitting up the obtained zero paths.
\end{proof}

\begin{proof}[Proof sketch for Theorem~\ref{zerothm}]
We follow very closely the strategy of the proof of Theorem~\ref{thm:2con}.
Lemma~\ref{largetanglelem} still yields a large tangle, and we can apply Theorem~\ref{thm:simpleflatwall} (we ignore the edge weights for this).
As before we then check the different outcomes.  
If the theorem yields  a large thoroughly labelled $K_t$ we also find, via Lemma~\ref{lem:rootedKn},
 a thoroughly labelled wall in there. This reduces outcome~(i) to outcome~(ii).
Outcome~(iii) still yields a hitting set as $H$ is $2$-connected. Outcome~(ii) can be reduced with Lemma~\ref{linkredlem}
to a wall with an $(\alpha,\beta)$-clean pair of linkages in series. 
Then we use Lemma~\ref{zeropathlem} to see that we also find a large such wall plus pair of linkages that are all zero. 
Finally, we apply Lemma~\ref{caseiiiclem} and note that each $H$-expansion we find 
will be a subdivision of the last wall (plus linkages) -- there all subdivided edges are zero.
\end{proof}

What happens if the host graphs are endowed with \emph{directed} group weights? 
First, as the weights now depend on the direction of the edges, we have to adjust
what it means for an $H$-expansion to be zero: summing up the weights of all of its edges
will not work. To make the definition a bit simpler, let us restrict ourselves to subcubic 
(labelled) graphs $H$. Then, an $H$-subdivision is \emph{zero}
if all its subdivided edges are zero paths. The same proof as before yields a version 
of Theorem~\ref{zerothm} for directed group weights. Indeed, as all our arguments depend
on zero paths the only critical point is Lemma~\ref{zeropathlem}, which holds 
for undirected as well as for directed group weights.

\section{Zero cycles with an infinite group}\label{infgroupsec}

Theorem~\ref{zerothm} may fail if its pre-conditions are weakened. If, for instance, we do not 
require $H$ to be $2$-connected then the conclusion will no longer follow: indeed, 
neither even $A$--$B$-paths, nor $A$-paths of a length divisible by~$6$ have the Erd\H os-P\'osa
property~\cite{BHJ18} -- without modularity constraints $A$--$B$-paths and $A$-paths have the property.

Similarly, Theorem~\ref{zerothm} may fail
if the edge weights come from an infinite group, rather than from a finite group. 
We demonstrate this with the group $\mathbb Z$ and zero cycles. 
Indeed, for every hitting set size $s$, we may choose $\ell>10s$
and construct a $\mathbb Z$-weighted graph $G$ that does not contain two disjoint zero cycles, but 
in which no vertex set of up to~$s$ vertices meets all zero cycles. 
For this, start with an $(\ell+1)\times (\ell+1)$-grid
to which for all rows, except for the last row, an additional edge between the first and the last vertex in the row are 
added; see Figure~\ref{intzerofig}. Let us call these additional edges \emph{wrap-around} edges.

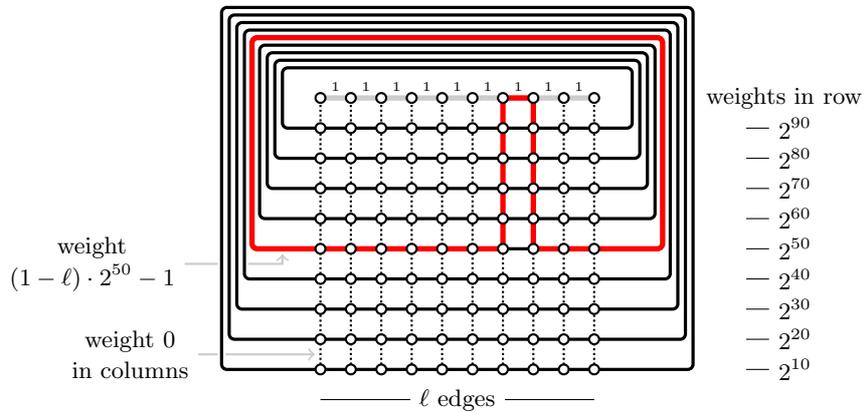
\begin{figure}[ht]
\centering
\begin{tikzpicture}
\tikzstyle{tinyvx}=[thick,circle,inner sep=0.cm, minimum size=1.3mm, fill=white, draw=black]

\def\gridsize{9}
\def\gridstep{0.4}
\pgfmathtruncatemacro{\sizeminone}{\gridsize-1}

\foreach \i in {0,...,\gridsize}{
  \draw[hedge] (0,\i*\gridstep) -- (\gridsize*\gridstep,\i*\gridstep);
  \draw[thick, densely dotted] (\i*\gridstep,0) -- (\i*\gridstep,\gridsize*\gridstep);
}

\node[align=center] (C) at (-2.5,0.5*\gridstep) {{\small  weight $0$}\\ {\small in columns}};
\draw[pointer] (C) -- ++(2.5,0);
\draw[white,line width=3pt] (-0.4-\gridsize*0.1,0) -- ++(0,0.5);

\node[align=center] (W) at (-3,3.5*\gridstep) {{\small weight}\\ {\small $(1-\ell)\cdot 2^{50}-1$}};
\draw[pointer] (W) -- ++(2.5,0) -- ++(0,0.5*\gridstep);
\draw[white,line width=3pt] (-0.4-\gridsize*0.1,3.1*\gridstep) -- ++(0,0.5);
\draw[white,line width=3pt] (-0.4-\gridsize*0.1+0.1,3.1*\gridstep) -- ++(0,0.5);
\draw[white,line width=3pt] (-0.4-\gridsize*0.1+0.2,3.1*\gridstep) -- ++(0,0.5);
\draw[white,line width=3pt] (-0.4-\gridsize*0.1+0.3,3.1*\gridstep) -- ++(0,0.5);

\foreach \j in {0,...,\sizeminone} {
  \pgfmathtruncatemacro{\invert}{\gridsize-\j}
  \draw[hedge,rounded corners=2pt] (\gridsize*\gridstep,\j*\gridstep) -- ++(0.4+\invert*0.1,0) -- ++(0,\invert*\gridstep+0.3+\invert*0.1) -- 
++(-\gridsize*\gridstep-0.8-2*\invert*0.1,0) -- ++(0,-\invert*\gridstep-0.3-\invert*0.1) -- (0,\j*\gridstep); 
}

\draw[line width=2pt, hellgrau] (0,\gridsize*\gridstep) -- ++(\gridsize*\gridstep,0);
\foreach \i in {1,...,\gridsize} {
  \node at (\i*\gridstep-0.5*\gridstep,\gridsize*\gridstep+0.15) {{\tiny $1$}};
}

\def\row{4}
\def\col{6}
\pgfmathtruncatemacro{\rowinvert}{\gridsize-\row}
\draw[line width=2pt, red,rounded corners=2pt] (0,\row*\gridstep) -- ++(\col*\gridstep,0) -- (\col*\gridstep,\gridsize*\gridstep) -- ++(\gridstep,0) 
-- (\col*\gridstep+\gridstep,\row*\gridstep) -- (\gridsize*\gridstep,\row*\gridstep) -- ++(0.4+\rowinvert*0.1,0) -- ++(0,\rowinvert*\gridstep+0.3+\rowinvert*0.1) -- 
++(-\gridsize*\gridstep-0.8-2*\rowinvert*0.1,0) -- ++(0,-\rowinvert*\gridstep-0.3-\rowinvert*0.1) -- cycle;

\foreach \i in {0,...,\gridsize}{
  \foreach \j in {0,...,\gridsize}{
    \node[tinyvx] (v\i\j) at (\i*\gridstep,\j*\gridstep) {};
  }
}

\node at (\gridsize*\gridstep+2.5,\gridsize*\gridstep) {{\small weights in row}};

\foreach \j in {0,...,\sizeminone}{
  \pgfmathtruncatemacro{\expo}{(\j+1)*10}
  \draw[thin] (\gridsize*\gridstep+2,\j*\gridstep) to ++(0.3,0) node[right]{{\small $2^{\expo}$}};
}

\draw[thin] (0,-\gridstep) to node[fill=white,midway]{{\small $\ell$ edges}} (\gridsize*\gridstep,-\gridstep);

\end{tikzpicture}
\caption{Zero cycles over $\mathbb Z$ do not have the Erd\H os-P\'osa property; a zero cycle is shown in red.}\label{intzerofig}
\end{figure}

All vertical edges of the grid receive weight~$0$ and all edges in the top row receive weight~$1$.
We define iteratively a weight $w_i$ for each edge in row $i$:
 for this choose a positive even integer
such that $w_i>\sum_{j=1}^{i-1}10\ell w_j$ (where we put $w_0=1$). 
The wrap-around edge between the first and last vertex in row~$i$
gets weight $-(\ell-1)w_i-1$.

Now let $C$ be a zero cycle. As $C$ cannot be contained in the union of all vertical edges,
and as the horizontal edges in the grid have positive weight, $C$ must contain a wrap-around
edge. Let $e$ be the wrap-around edge in $C$ with largest weight, and assume that $e$ 
is incident with row $i$. By the choice of weights, $C$ can then not contain any edges 
of rows $j>i$, except for edges from the top row -- indeed, such an edge would have a weight 
so large that it could never be canceled by the weights of the wrap-around edges in $C$. 
Moreover, to balance the weight of $e$, the cycle $C$ needs to contain exactly $\ell-1$ edges from row~$i$.
Due to parity, $C$ either must contain an odd number of wrap-around edges or an odd number of edges from 
the top row. One may deduce that $C$ cannot contain three or more wrap-around edges (as then 
always exactly $\ell-1$ other edges from that row must be in $C$ as well) -- thus $C$ can only 
contain edges from row $i$, plus an odd number of edges from the top row. As a consequence, 
every zero cycle traverses the grid from left to right and picks up at least one edge from the top row
in the process. This shows that there cannot be two disjoint zero cycles. On the other hand, 
as each row and column gives rise to a zero cycle, see Figure~\ref{intzerofig}, it is obvious
that $s$ vertices cannot meet all zero cycles.

Note: the same example works for directed weights, as long as all horizontal edges are directed 
from left to right, and the wrap-around edges from the last column to the first column.
The argumentation becomes more tedious, though.

\bibliographystyle{amsplain}
\bibliography{erdosposa}

\vfill

\small
\vskip2mm plus 1fill
\noindent
Version \today{}
\bigbreak

\noindent
Henning Bruhn
{\tt <henning.Bruhn@uni-ulm.de>}\\
Institut f\"ur Optimierung und Operations Research\\
Universit\"at Ulm\\
Germany\\

\noindent
Felix Joos
{\tt <f.joos@bham.ac.uk>}\\
School of Mathematics\\ 
University of Birmingham\\
United Kingdom\\

\noindent
Oliver Schaudt
{\tt <schaudt@mathc.rwth-aachen.de>}\\
Department of Mathematics\\
RWTH Aachen University\\
Germany\\

\end{document}